\documentclass[a4paper,reqno]{amsart}
\title[Indecomposable solutions with permutation brace of size $p^3$]{Indecomposable solutions with permutation brace of size $p^3$ and permutation braces of small sizes}
\date{}
\author{Andrew Darlington}
\address{Department of Mathematics, Vrije Universiteit Brussel, Pleinlaan 2,
1050 Brussel, Belgium}
\email{andrew.darlington@vub.be}

\author{Magdalena Wiertel}
\address{Department of Mathematics, Vrije Universiteit Brussel, Pleinlaan 2,
1050 Brussel, Belgium}
\email{m.wiertel@mumuw.edu.pl}

\thanks{For the purpose of open access, the authors have applied a CC BY 
public copyright license to any Author Accepted Manuscript version arising.
\newline
\indent Data Access Statement: Data sharing is not applicable to this article 
as no datasets were generated or analysed in this research.
}

\usepackage[utf8]{inputenc}
\usepackage{amsmath}
\usepackage{amsfonts}
\usepackage{amssymb}
\usepackage{amsthm}
\usepackage{pifont} 
\usepackage{enumerate} 
\usepackage{enumitem}
\usepackage{wasysym} 
\usepackage{mathrsfs}
\usepackage{xfrac} 
\usepackage{array}
\usepackage{siunitx}
\usepackage{mathtools}
\usepackage{framed}
\usepackage{pifont}
\usepackage[utf8]{inputenc}
\usepackage[english]{babel}
\usepackage[utf8]{inputenc}
\usepackage{tikz,tkz-euclide}
\usepackage{xcolor,graphicx}
\usepackage{faktor}
\usepackage{xfrac}
\usepackage{float}
\usepackage{tikz-cd}
\usepackage{rotating}
\usepackage{hyperref}
\usepackage{hhline}

\newcommand{\id}{\operatorname{id}}
\makeatletter
\def\bign#1{\mathclose{\hbox{$\left#1\vbox to8.5\p@{}\right.\n@space$}}\mathopen{}}
\def\Bign#1{\mathclose{\hbox{$\left#1\vbox to11.5\p@{}\right.\n@space$}}\mathopen{}}

\newtheorem{theorem}{Theorem}[section]
\newtheorem{proposition}[theorem]{Proposition}
\newtheorem{lemma}[theorem]{Lemma}

\newtheorem{corollary}[theorem]{Corollary}

\newtheorem{definition}[theorem]{Definition}

\theoremstyle{definition}
\newtheorem{remark}[theorem]{Remark}

\newtheoremstyle{solutions}
  {3pt}{3pt}
  {\normalfont}
  {}
  {\bfseries}
  {.}
  {.5em}
  {\thmname{#1}\thmnote{ #3}}

\theoremstyle{solutions}
\newtheorem*{solutions}{Solutions}
\newtheorem*{solution}{Solution}

\newcommand{\Aut}{\mathrm{Aut}}
\newcommand{\Stab}{\mathrm{Stab}}
\newcommand{\Orb}{\mathrm{Orb}}
\newcommand{\bb}{\mathbf{b}}
\newcommand{\x}{\mathbf{x}}
\newcommand{\z}{\mathbf{z}}
\newcommand{\bc}{\mathbf{c}}
\newcommand{\Z}{\mathbb{Z}}
\newcommand{\F}{\mathbb{F}}
\newcommand{\m}{\operatorname{mod}}

\begin{document}

\bibliographystyle{amsalpha}

\begin{abstract}
    Using the construction of Bachiller, Ced\'o and Jespers, we give a complete classification of the indecomposable involutive set-theoretic solutions to the Yang--Baxter equation whose permutation brace has size $p^3$, where $p$ is an odd prime. We also give an algorithm for systematically producing all indecomposable involutive solutions with a given permutation brace, and use this to enumerate these with the permutation brace of size up to~107.\\
\end{abstract}

\maketitle

\section{Introduction}
The Yang--Baxter equation (YBE) arises across mathematics and physics in many different forms and contexts. Owing its name to Chen-Ning Yang and Rodney Baxter, who themselves studied the equation for very different reasons (Yang in the context of quantum integrable systems \cite{Yan67} and Baxter who used it to solve the eight vertex model \cite{Bax72}), the equation has increasingly drawn attention and gained popularity to become a fundamental object in several fields.


In 1992, Drinfeld \cite{Dri92} proposed the study of solutions of the YBE of a combinatorial nature, allowing one to employ
rich algebraic tools 
(such as those coming from group and ring theory) in order to tackle this problem.

A set-theoretic solution of the YBE is a pair $(X, r)$, where $X$ is a non-empty set and $r: X\times X \rightarrow X\times X$ is a map such that 
\begin{equation}(r\times \id)(\id \times r )(r\times \id) = (\id \times r )(r\times \id) (\id \times r ) \tag{YBE}\label{YBE}
\end{equation}
holds on $X^3$.
Describing broad families of the set-theoretic solutions remains one of the main goals of the research in this area.
Due to the vast amount of such solutions even on sets of small sizes,
it is natural to start with classes that serve as building blocks of arbitrary solutions. One of such important families are the finite
non-degenerate
involutive
indecomposable solutions, \cite{ESchSol} (we will abbreviate ``finite non-degenerate involutive solution'' as simply ``solution'', unless stated otherwise). As their name suggests, these solutions are minimal in some sense (namely, they cannot be `decomposed' into non-trivial solutions of smaller size).



Some results concerning the decomposability of
solutions have been established. Methods have also been developed showing how a certain permutation associated to a solution, known as the \textit{diagonal map}, provides a lot of information about the decomposability, \cite{LV24, R05, RV22}. Moreover, it was shown in \cite{CCP19} how to extend indecomposable solutions to create solutions on larger sets.



An alternative viewpoint from which to study indecomposable solutions comes from looking at the associated permutation group, first introduced in \cite{ESchSol}. Namely, for any solution $(X, r)$ one can define a subgroup,
$\mathcal{G}(X,r)$, of $\mathrm{Perm}(X)$
that acts naturally on the underlying set. Indecomposable solutions can be characterised as the ones for which $\mathcal{G}(X,r)$ acts transitively.
It can further be shown that the permutation group of a solution carries the richer structure of a \textit{brace}, often called the \textit{permutation brace}. Introduced in \cite{Rump07} and reformulated in \cite{CJO14}, braces arose from the observation that radical rings produce solutions of the YBE. Since then, braces and their generalisations have been considered as one of the most important tools in the area. In particular, the permutation brace, has become a fruitful object with which to study indecomposable solutions, see \cite{C25, COk23, Rump21}.



So far few families of indecomposable solutions have been constructed, but literature does exist in some cases. Firstly, it has been proved in \cite{ESchSol} that for any prime $p$, up to isomorphism, there is only one indecomposable (involutive) solution of size $p$. All indecomposable solutions (up to isomorphism) have been classified up to size 11, \cite{AMV22}. 
The indecomposable solutions of size $p^2$, where $p$ is prime, have been characterised in \cite{DPT25}. Indecomposable solutions of multipermutation level $2$ (this property in some sense measures how close to the trivial solution they are) have been described in \cite{JPZ-D21, JP22}. As mentioned above, (transitive) permutation braces have been used in many settings: indecomposable solutions with permutation braces of size $p^2$ and several of size $p^2q$ (for distinct primes $p$ and $q$) have been described in \cite{Ram23}; those with cyclic permutation brace have been given in \cite{Rump21} and \cite{JPZ-D}; and the special case for which the permutation brace acts regularly (that is it additionally acts with trivial stabilisers) has been extensively studied in, for example, \cite{Cas23, CR24}.




Of particular interest in our paper is the method developed in \cite[Theorem 3.1]{BCJ16}, which allows one to obtain all solutions with a given permutation brace $B$.
Every brace is a set equipped with two group structures linked by the action of one of them on the other, called the $\lambda$-action. All solutions are constructed by decomposing $B$ into disjoint $\lambda$-orbits, along with special families of subgroups of each of those, and gluing the corresponding cosets of $B$ together to create a `patchwork'-type solution of the aforementioned form. The method also makes it possible to determine whether the obtained solution will be indecomposable before explicitly constructing it. Namely, the obtained solution will be indecomposable if and only if it is built from a single $\lambda$-orbit. This has the upshot of greatly simplifying the construction, meaning that classification schema aimed at producing all indecomposable solutions with a given (family of) permutation brace(s) become very approachable.


In the present paper, we apply this construction, utilising the brace classification given in
\cite{Bac15}, to describe all indecomposable solutions, up to isomorphism, with permutation brace of size $p^3$ for any odd prime $p$.
In doing so, we compute the brace automorphism groups in some specific cases (noting, however, that all automorphism groups of braces of order $p^3$ are computed in \cite{NZ18}).
The obtained solutions have size $p^3$ or $p^2$. In the latter case,
we recover some of the solutions of size $p^2$ described in \cite{DPT25}.  
The enumeration of the constructed solutions can be summarized as follows.
\begin{theorem}
Let $p$ be an odd prime. Up to isomorphism, there are $(p^3 + 3p)/4$ and $(p^3 + p^2 +11p+3)/4$ solutions of size $p^2$ and $p^3$, respectively, with permutation brace of size $p^3$.

\enlargethispage{\baselineskip}

Moreover, denote by $n_{2}(p, B)$ and $n_{3}(p, B)$ the number of indecomposable solutions of size $p^2$ and $p^3$, respectively, with the permutation brace $(B, + , \cdot)$. Then
\[n_{2}(p, B) = \begin{cases}
p  & \textrm{ if } (B, +) \cong \mathbb{Z}_{p}^3,\\
p(p^2-1)/4  & \textrm{ if }(B, + ) \cong \mathbb{Z}_{p}^2\times \mathbb{Z}_{p}
\end{cases}\]
and
\[
n_{3}(p, B) = \begin{cases}
p+1 &\textrm{ if } (B, + ) \cong \mathbb{Z}_{p}^3,\\
1 + (p+1)(p^2 - 1)/4 &\textrm{ if } (B, + ) \cong \mathbb{Z}_{p}^2\times \mathbb{Z}_{p},\\
2p-1 &\textrm{ if } (B, + ) \cong \mathbb{Z}_{p^{3}}.
\end{cases}\]
\end{theorem}
All obtained indecomposable solutions are multipermutation and thus we give their multipermutation level. 
\begin{corollary}
Let $p$ be an odd prime. Up to isomorphism, there are
\[2p -  1+ (p+1)(p^2 -1)/4\]
indecomposable solutions of multipermutation level $3$
and
\[2p + 1 +p(p^2-1)/4\]
indecomposable solutions of multipermutation level $2$ with permutation brace of size~$p^3$.
\end{corollary}



Finally, we present a computer algorithm, implemented in {\sc Gap} \cite{GAP4}, for computing indecomposable solutions (up to isomorphism) with permutation braces of small size, utilising the database of braces
obtained in \cite{GV17}. We employ the construction from \cite{BCJ16} to describe all indecomposable solutions up to isomorphism, with permutation brace of sizes up to $107$, except for sizes $32$, $64$, $80$, $81$ and $96$. The obtained database is available at 
\cite{GITPAGE}. We remark that that the classification of all indecomposable solutions with permutation brace of size $2^3$ also follows from our computational results.



\section{Preliminaries}

Let us start by introducing some notation. If $(X, r)$ is a solution of the \eqref{YBE}, write
\[r(x, y) =(\sigma_x(y), \tau_y(x))\]
for any $x, y\in X$. We will assume that $X$ is finite. A solution is called \textit{non-degenerate} if both $\sigma_x$ and $\tau_x$ are bijective for all $x\in X$, and \textit{involutive} if $r^2=\mathrm{id}_{X \times X}$. In this paper, by a ``solution'' we always mean a ``finite non-degenerate and involutive solution''. Note that the condition that $r^2=\mathrm{id}_{X \times X}$ trivially implies that $r$ is bijective. 

Two solutions $(X, r), (Y, s)$ are called isomorphic if there is a bijective map $f:X \to Y$ such that $(f \times f)r=s(f \times f)$.

To a solution $(X,r)$ we can associate the following group, known as the \textit{permutation group}:
\[\mathcal{G}(X,r)=\langle \sigma_x \mid x \in X \rangle,\]
with group operation given by composition. This is clearly a subgroup of $\mathrm{Perm}(X)$, and hence admits a natural action on $X$. 


A solution $(X,r)$ is \textit{indecomposable} if there does not exist any pair of non-empty subsets $Y,Z \subset X$ with $Y \cap Z=\emptyset$, where $(Y,r|_{Y \times Y})$ and $(Z,r|_{Z \times Z})$ are both solutions. It can be shown that $(X,r)$ is indecomposable if and only if $\mathcal{G}(X,r)$ is transitive on $X$.


\begin{definition}
A triple $(B, +, \cdot)$, where $(B, +)$ is an abelian group and $(B, \cdot)$ is a group, is a \textit{brace} if
$$ a\cdot (b + c) = a\cdot b - a + a\cdot c $$
is satisfied  for all $a, b, c\in B$. 
\end{definition}
\noindent The groups $(B,+)$ and $(B,\cdot)$ are known as the \textit{additive} and \textit{multiplicative} groups of $(B,+,\cdot)$ respectively. The isomorphism class of $(B,+)$ is often referred to as the \textit{type} of $(B,+,\cdot)$.
Let $(B,+,\cdot)$ be a brace, then one can define the following group homomorphism:
\begin{align*}
    \lambda: (B,\cdot) \to \Aut(B,+),\quad
    a &\mapsto \lambda_a,
\end{align*}
where $\lambda_a(b)=-a+a \cdot b$ for all $a,b \in B$. This therefore gives an action (by automorphisms) of the multiplicative group on the additive group of the brace. For an element $x \in B$, we will denote by $\Orb(x)$ and $S(x)$ the $\lambda$-orbit and $\lambda$-stabiliser of $x$ in $(B,\cdot)$ respectively.
The \textit{socle}, $\mathrm{Soc}(B)$, of a brace $(B, +, \cdot)$ is
\[\mathrm{Soc}(B)=\ker(\lambda)=\{a \in B \mid -a+a\cdot b=b \;\; \forall b \in B\}.\]
The socle may also be identified as the intersection of all $\lambda$-stabilisers, that is
\[\mathrm{Soc}(B)=\bigcap_{b \in B}S(b).\]

One may endow $\mathcal{G}(X,r)$ with the structure of a brace in a following way.

\begin{remark}
    Let $(X,r)$ be a solution and $(\mathcal{G},\circ)=\mathcal{G}(X,r)$ (with $\circ$ denoting the composition operation) its associated permutation group. Then there is a unique operation $+:\mathcal{G} \times \mathcal{G} \to \mathcal{G}$ satisfying 
    \[\sigma_x + \sigma_y=\sigma_x \circ \sigma_{\sigma^{-1}_x(y)} \;\; \forall x,y \in X\]
    such that $(\mathcal{G},+,\circ)$ is a brace.
    
\end{remark}

We now state the construction of \cite{BCJ16}, specialised to the case of indecomposable solutions. 
\begin{theorem}[Bachiller -- Ced\'o -- Jespers]\label{BCJ}
Let $(B, +, \cdot)$ be a brace. Suppose that there is an $x\in B$ such that $B=\langle \Orb(x)\rangle_+$ and $K \leqslant S(x)$ is core-free in $(B,\cdot)$. Let $X=B/K$. Then the map $r: X\times X \rightarrow X\times X$ given by
\begin{equation}\label{solution}
    r(b\cdot K,c\cdot K)=\left(\lambda_b (x)\cdot c\cdot K,\lambda_{\lambda_b(x)\cdot c}(x)^{-1}\cdot b\cdot K\right)
\end{equation}
gives an indecomposable solution on $X$ such that $\mathcal{G}(X, r)\cong B$ as braces. Moreover every indecomposable solution $(Y, r)$ with $\mathcal{G}(Y, r)\cong B$ arises in this way.

Assume that $(X_1, r_1),(X_2, r_2)$ are solutions associated to $x, x'\in B$ respectively with $K \leqslant S(x)$ and $K' \leqslant S(x')$. Then $(X_1, r_1) \cong (X_2, r_2)$ if and only if there exist $\psi\in \Aut((B, +, \cdot))$ and $z\in B$ such that $\psi(x) = \lambda_z(x')$ and $\psi(K)= zK'z^{-1}$.
\end{theorem}
In light of this, for the pairs $(x,K)$ and $(x',K')$ as in the theorem, we will often denote by
\[(x,K) \sim (x',K')\]
to mean that the associated solutions are isomorphic (that is $(X_1,r_1) \cong (X_2,r_2)$). For a solution \eqref{solution} constructed from a chosen $x$ and chosen $K$ we will say that it is generated by $(x, K)$. 


The following general remark will be useful to find the isomorphism classes of the constructed solutions.





    \begin{remark}\label{elim-elts} Let $(B, +, \cdot)$ be a brace and suppose $x,x' \in B$ are such that $(x, \{ 1\}) \sim (x', \{ 1\})$, that is there is a $z \in B$ and a $\psi \in \Aut(B,+,\cdot)$ such that $\psi(x')=\lambda_z(x)$. We remark that $S(x')=\psi^{-1}(zS(x)z^{-1})$. Indeed, if $c \in S(x)$, then
        \begin{align*}
            \lambda_{\psi^{-1}(zcz^{-1})}(x')&=\lambda_{\psi^{-1}(zcz^{-1})}(\psi^{-1}(\lambda_z(x)))\\
            &=\psi^{-1}(\lambda_{zcz^{-1}}(\lambda_z(x)))\\
            &=\psi^{-1}(\lambda_z(x))=x'.
        \end{align*}
        We then note that if $K \leqslant S(x)$ is core-free, then $\psi^{-1}(zKz^{-1}) \leqslant S(x')$ is core-free. Therefore, for any (core-free) $K \leqslant S(x)$, there exists a (core-free) $K' \leqslant S(x')$ such that $(x,K) \sim (x',K')$ (and vice versa).
    \end{remark}
Finally, we will give the multipermutation level of each of the obtained solutions. Recall that for a solution $(X, r)$ one may define the following equivalence relation on $X$:
\[x \approx y \Leftrightarrow \sigma_x=\sigma_y.\]
This induces a solution $\mathrm{Ret}(X, r) = (X/\approx, \overline{r})$, known as the \textit{retraction} of $(X,r)$, on the set of equivalence classes $X/\approx$ such that $\overline{r}([x], [y]) = ([\sigma_{x}(y)], [\tau_{y}(x)])$. Noting that this process can be iterated, for $k \geqslant 1$, we define recursively
\begin{align*}
    \mathrm{Ret}^{1}(X, r) = \mathrm{Ret}(X, r),\quad
    \mathrm{Ret}^{k+1}(X, r) = \mathrm{Ret}(\mathrm{Ret}^{k}(X, r)).
\end{align*}
We say that $(X, r)$ is \textit{multipermutation} if there is a $k\geqslant 1$ such that $\mathrm{Ret}^{k}(X, r)$ is a solution on the set of size $1$. The smallest such $k$ is then called the \textit{multipermutation level} of a solution $(X, r)$.
We make the following observation: 
\begin{remark}
Every indecomposable solution $(X, r)$ such that $\mathcal{G}(X, r)$ is of size $p^3$ is of multipermutation level at most $3$.
\end{remark}

\begin{proof}
From \cite{Bac15} it follows that for any odd prime $p$ the socle of every brace $\mathcal{G}(X, r)$ of size $p^3$ is non-trivial. Then, as $\mathcal{G}(X, r)/\mathrm{Soc}(\mathcal{G}(X, r))$ is right nilpotent, it follows that $\mathcal{G}(X, r)$ is also right nilpotent, \cite{CSV19}. From \cite{GI18} it thus follows that the solution $(X, r)$ is multipermutation. Finally, as $|\textrm{Ret}(X, r)|$ is a divisor of $|X|$, it is clear that the multipermutation level is at most $3$.  
\end{proof}


\section{Indecomposable solutions with permutation brace of size \texorpdfstring{$p^3$}{p3}}\label{IndSols_p3}

Fix $p$ to be an odd prime. In this section, we will consider each brace $(B,+,\cdot)$ of size $p^3$ in turn and construct all the indecomposable solutions, up to isomorphism, with permutation brace $(B,+,\cdot)$.

We note that we do not focus on the case $p=2$ in this section, as the braces of order 8 slightly differ in structure to those of order $p^3$ for $p>2$, and the indecomposable solutions with permutation brace of size $8$ are easily constructed using (for example) the algorithm given in Section \ref{alg}.



The braces of order $p^3$ have been classified in \cite{Bac15} with respect to their underlying additive groups, which are of one of three types: $\Z_{p^3},\Z_{p^2} \times \Z_p$ and $\Z_p^3$.
We will use this classification here.
Throughout the text, we follow notations from that paper without further comments. Namely, multiplication without a dot denotes the standard multiplication in the ring $\Z_{p^n}$ for some relevant $n$.

We on occasion abuse notation by considering elements of $\Z_p$ inside $\Z_{p^2}$, which would typically be not well-defined. However we note that, in all such cases, the elements are actually considered in the ideal $p\Z_{p^2}$, in which case there is a well-defined homomorphism. That is, considering an element $x \in \Z_p$ inside $p\Z_{p^2}$ amounts to considering $x$ under the image of the map
\[\Z_p \to p\Z_{p^2}\]
given by $x \mapsto px$.
Moreover, let $(B,+,\cdot)$ be a brace and $b \in B$; we will write $\binom{b}{2}$ to mean $b(b-1)/2$. In particular, we have ${\binom{b}{2}}=0$ for $b=0,1$.


As remarked earlier, there is only one indecomposable solution of prime size (up to isomorphism) and its permutation brace is of the same size. We thus obtain the following consequence:


\begin{remark}\label{rem:size-corefree}
If $(B,+,\cdot)$ is a brace of size $p^3$ for an odd prime $p$, all core-free subgroups $K$ of $(B, \cdot)$ from Theorem~\ref{BCJ} are of size $1$ or $p$. 
\end{remark}

\subsection{Braces of cyclic type}
In this section, we investigate the indecomposable solutions admitting a permutation brace of cyclic type. These solutions were already characterized in \cite{JPZ-D}, but we present them for the sake of completeness.
Note first that any $x \in (B,+)$ of (additive) order $p^3$ has $\langle\Orb(x)\rangle_+=B$ and conversely, if we have $\langle\Orb(x)\rangle_+=B$, then $x \in (B,+)$ has to be of order $p^3$. By \cite{Bac15}, we also have that $(B,\cdot) \cong \Z_{p^3}$, hence the only core-free subgroup of $(B,\cdot)$ is $\{1\}$. In particular in this case all solutions will be of size $p^3$.

There are three braces of cyclic type. These are given by the following $\lambda$-maps for all $b, x \in \mathbb{Z}_{p^3}$:
\begin{enumerate}
    \item $\lambda_b(x)=x$ \text{(trivial brace)},
    \item $\lambda_b(x)=x(1+pb)$,
    \item$\lambda_b(x)=x(1+p^2b)$.
\end{enumerate}
We will denote these by $(A_1,+_1,\cdot_1),(A_2,+_2,\cdot_2)$ and $(A_3,+_3,\cdot_3)$, respectively. We abuse notation by ignoring the indices of the brace operations. The following is clear:
\begin{solution}[associated to $A_1$]
    Up to isomorphism the unique indecomposable solution associated to the brace $(A_1,+,\cdot)$ is given by $(A_1,r)$, with
    \[r(b,c)=(c+1,b-1)\]
    for all $b,c \in A_1$.
\end{solution}
We now move onto the non-trivial cyclic braces. Let us describe their automorphisms groups.
\begin{proposition}
    Let $n=2,3$. Then
    \[\Aut(A_n,+,\cdot) \cong \Z_{p^{n-1}}.\]
\end{proposition}
\begin{proof}
    It is clear that $\Aut(A_n,+) \cong \Z_{p^3}^\times$, where $\psi \in \Aut(A_n,+)$ if and only if $\psi(x)=kx$ for some $k \in \Z_{p^3}^\times$. Suppose now that $\psi \in \Aut(A_n,+,\cdot)$, so for all $x_1,x_2 \in A_n$, we must have
    \[k(x_1\cdot x_2)=\psi(x_1\cdot x_2)=\psi(x_1)\cdot\psi(x_2)=(kx_1)\cdot(kx_2),\]
    and in particular, $kp^{n-1}x_1x_2=k^2p^{n-1}x_1x_2$. Thus $k \equiv 1 \pmod{p^{4-n}}$, and so
    \[k \in \{mp^{4-n}+1   \mid0 \leqslant m \leqslant p^{n-1}-1 \}\]
    giving $p^{n-1}$ choices for $k$. As $\Z_{p^3}^\times$ is cyclic, this completes the proof.
\end{proof}
\begin{proposition}
    Let $n=2,3$. Then there are $p-1$ non-isomorphic indecomposable solutions with permutation brace $(A_n,+,\cdot)$. All such solutions have size~$p^3$.
\end{proposition}
\begin{proof}
    As mentioned, we have that $\langle\Orb(x)\rangle_+ = A_n$ if and only if $x$ has order $p^3$ in $(A_n,+)$. Suppose $x,x'\in A_n$ are two such elements. Then they generate the same solution if there is a $\psi \in \Aut(A_n,+,\cdot)$ and a $z \in A_n$ such that $\psi(x)=\lambda_z(x')$. That~is,
    \[(mp^{4-n}+1)x=x'(1+p^{n-1}z)\]
    for some $0 \leqslant m \leqslant p^{n-1}-1$ and $z \in A_n$, and in particular, we have $x'=(mp^{4-n}+1)(1+p^{n-1}z)^{-1}x$. It can then be shown that the set
    \[\{(mp^{4-n}+1)(1+p^{n-1}z)^{-1}   \mid z \in A_n, 0 \leqslant m \leqslant p^{4-n}-1\}\]
    has cardinality $p^2$, and hence each $x \in A_n$ of additive order $p^3$ generates an indecomposable solution isomorphic to that generated by $p^2-1$ other elements. As there are $p^2(p-1)$ choices for $x$, we therefore get that there are $p^2(p-1)/p^2=p-1$ non-isomorphic indecomposable solutions with permutation brace $(A_n,+,\cdot)$.
\end{proof}
We can now write down the explicit indecomposable solutions generated by the non-trivial cyclic braces.
\begin{solutions}[associated to $A_2$ and $A_3$]
    Let $n=2,3$ and $1 \leqslant x \leqslant p-1$. Then, up to isomorphism, the $p-1$ indecomposable solutions associated to the brace $(A_n,+,\cdot)$ are given by $(A_n, r)$, where 
    \begin{align*}
        \sigma_b(c) &=  x + c + p^{n-1}x(b+c) + p^{2(n-1)}xbc, \\
        \tau_c(b) &= b - x -p^{n-1} x(b+c) -p^{2(n-1)}x(b-x)(c+x),
  \end{align*}
  for all $b,c \in A_n$.
\end{solutions}
The solutions are of multipermutation level 3 and 2 for $n=1$ and $n=2$ respectively.

\begin{corollary}
The number of non-isomorphic solutions of size $p^{3}$ with cyclic permutation brace is $2p-1$.    
\end{corollary}
Note that this agrees with Proposition 4.1 of \cite{JPZ-D}.
\subsection{Braces of type \texorpdfstring{$\Z_p \times \Z_{p^2}$}{Zp2 x Zp}}
In this section, we investigate the indecomposable solutions admitting a permutation brace of type $\Z_p \times \Z_{p^2}$. In what follows, let $0 \leqslant \mu \leqslant (p-1)/2$, $0 \leqslant a \leqslant p-1$ and $\varepsilon \in \{1\} \cup \F_p\setminus\F_p^2$. We will denote an arbitrary element $\x \in \Z_p\times \Z_{p^2}$ by $(x_1, x_2)$.

By \cite{Bac15}, there are $(3p^2 + 10p + 25)/4$ braces of type $\Z_p \times \Z_{p^2}$ up to isomorphism. Our first proposition shows that $(p^{2} + 6p +23)/4$ of these cannot be the permutation brace of an indecomposable solution.
\begin{proposition}\label{prop:no-ind-sols1}
    Let $(B,+,\cdot)$ be a brace of type $\Z_p \times \Z_{p^2}$ with one of the following $\lambda$-maps for all $\bb, \x \in \mathbb{Z}_p\times \mathbb{Z}_{p^2}$
    \begin{enumerate}
        \item $\lambda_\bb(\x)=(x_1,x_2 + p(x_1(\varepsilon b_1+\mu b_2)+x_2b_2))$,
        \item $\lambda_\bb(\x)=(x_1,x_2+p(\eta b_2x_1+b_1x_2))$ for $\eta\in \{1, \ldots, p-1\}$,
        \item $\lambda_\bb(\x)=(x_1,x_2+p(b_1x_2+b_2x_2-b_2x_1))$,
        \item $\lambda_\bb(\x)=(x_1,x_2(1+pb_1))$,
        \item $\lambda_\bb(\x)=(x_1,x_2(1+pb_2))$,
        \item $\lambda_\bb(\x)=(x_1,x_2+px_1b_2)$,
        \item $\lambda_\bb(\x)=(x_1,x_2+px_1b_1)$,
        \item $\lambda_\bb(\x)=(x_1,x_2+px_1(b_1+ b_2))$,
        \item $\lambda_\bb(\x)=\x$.
    \end{enumerate}
    Then there are no indecomposable solutions with $(B,+,\cdot)$ as a permutation brace.
\end{proposition}
\begin{proof}
    Brief calculations show that every element $\x$ in any of the braces above has $\langle \Orb(\x)\rangle_+ \leqslant H$, where
    \[H=\langle (x_1, x_2 + pm)\;   \mid \; 0 \leqslant m \leqslant p-1\rangle_+ = \{ (kx_1, kx_2 + pm') \;   \mid \; k\in \Z , 0 \leqslant m' \leqslant p-1\}.\]
    If $(0,1) \in H$, then $(0,1)=(kx_1,kx_2+pm')$ for some $k,m'$. If $p \nmid k$ then $x_1=0$ and $H \subseteq \{0\} \times \Z_{p^2} \subsetneq B$. Otherwise $H \subseteq \Z_p \times p\Z_{p^2} \subsetneq B$. 
\end{proof}
There are three (families of) braces of type $\Z_p \times \Z_{p^2}$ which are not included in the list of Proposition \ref{prop:no-ind-sols1}. These are given by the following $\lambda$-maps for all $\bb, 
\x\in \mathbb{Z}_{p}\times \mathbb{Z}_{p^2} $:
\begin{enumerate}
    \item $\lambda_\bb(\x)=(x_1+x_2b_2,x_2)$,
    \item $\lambda_\bb(\x)=\left(x_1+b_2x_2, x_2 + pab_2x_1+px_2\left(\varepsilon b_1+(a-\varepsilon){\binom{b_2}{2}}\right)\right)$,
    \item $\lambda_\bb(\x)=\left(x_1+x_2b_2,x_2+p\varepsilon\left(x_1b_2+{\binom{b_2}{2}}x_2\right)\right)$.
\end{enumerate}
We will denote these by $(B_1,+_1,\cdot_1),(B_2,+_2,\cdot_2)$ and $(B_3,+_3,\cdot_3)$, respectively.

We abuse notation by ignoring the indices of the brace operations. We will show that each of these braces lead to indecomposable solutions.

\subsubsection{Solutions related to \texorpdfstring{$(B_1,+_1,\cdot_1)$}{B1}}
Denote by $(B_1,+_1,\cdot_1)$ the brace associated~to
\[\lambda_\bb(\x)=(x_1+x_2b_2,x_2)\]
for all $\bb, \x \in \mathbb{Z}_{p}\times\mathbb{Z}_{p^2}$. 
It is then clear that $\langle \Orb(\x)\rangle_+ = B_1$ if and only if $x_2$ has additive order $p^2$ in $\Z_{p^2}$. Moreover, it is shown in \cite{Bac15} that $(B_1,\cdot) \cong \Z_p\times \Z_{p^2}$, so the only core-free subgroup of $(B_1,\cdot)$, hence also of $\Stab(\x)$, is trivial.
Therefore all indecomposable solutions with permutation brace $(B_1,+,\cdot)$ have size $p^3$. We now compute $\Aut(B_1,+,\cdot)$.
\begin{proposition}
    $\Aut(B_1, +, \cdot) $ is isomorphic to 
      \[\left\{ \begin{pmatrix}
          \alpha & \beta \\ 0 & \delta 
      \end{pmatrix}   \middle| \; \alpha,\beta\in \Z_p,\delta \in \Z_{p^2},  \alpha \not\equiv 0 \m{p}, \; \alpha \equiv \delta^2 \m{p} \right\}\]
    with the multiplication given by matrix multiplication in $\Z$ followed by a reduction modulo $p$ on the first row and a reduction modulo $p^2$ on the second one.
    
\end{proposition}
    \begin{proof}
       Note, following~\cite[page 3592]{Bac15}, that $\Aut(B_1,+)$ is isomorphic to 
       \[\left\{ \begin{pmatrix}
          \alpha & \beta \\ p\gamma & \delta 
      \end{pmatrix}   \middle| \gamma, \delta \in \Z_{p^2}, \alpha, \beta\in \Z_p,  \alpha\delta\not\equiv0\m p\right\},\]
      with the multiplication, which is the matrix multiplication in $\Z$, followed by a reduction modulo $p$ on the first row and a reduction modulo $p^2$ on the second one. Take $\psi \in \Aut(B_1, +)$ corresponding to the matrix $\begin{pmatrix}
          \alpha & \beta \\ p\gamma & \delta 
      \end{pmatrix}$. Then $\psi$ is an automorphism of $(B_1, +, \cdot)$ if and only if $$\psi(x_1, y_1) \cdot \psi(x_2, y_2) = \psi((x_1, y_1)\cdot (x_2, y_2))$$ for all $(x_1, y_1), (x_2, y_2) \in \Z_p\times \Z_{p^2}$. Direct computation shows that this is equivalent to $(\alpha - \delta^2, p\gamma) = (0, 0)$ in $\Z_p\times \Z_{p^2}$. The assertion follows. 
    \end{proof}
    \begin{proposition}
        Up to isomorphism, there is a unique indecomposable solution with permutation brace isomorphic to $(B_1,+,\cdot)$.
    \end{proposition}
    \begin{proof}
    Given two elements $\x$ and $\x'$ in $B_1$ such that
    \[\langle\Orb(\x)\rangle_+=\langle\Orb(\x')\rangle_+=B_1,\]
    consider the element $\psi \in \Aut(B_1,+,\cdot)$ given by
    \[\begin{pmatrix}
          \alpha^2 & (x_1' - \alpha^2x_1)x_2^{-1} \\ 0 & \alpha 
      \end{pmatrix},\]
    with $\alpha=x_2'x_2^{-1}$. Then $\alpha \not\equiv 0 \pmod{p}$ and $\psi(\x)=\x'=\lambda_{\mathbf{0}}(\x')$, so the solution associated to $\x$ is isomorphic to the solution associated to $\x'$.
    \end{proof}
     Taking, without loss of generality, $x=(0,1) \in B_1$, we may now state the following:
    \begin{solution}[associated to $B_1$]
    Up to isomorphism the unique indecomposable solution associated to the brace $(B_1,+,\cdot)$ is given by $(B_1, r)$, where 
    \begin{align*}
        \sigma_\bb(\bc) &= (b_2 + c_2 + c_1, c_2 + 1),\\
        \tau_{\bc}(\bb) &= (b_1 - b_{2} -c_2, b_2 - 1),
    \end{align*}
  for all $\bb,\bc \in B_1$.
\end{solution}    
This solution is of multipermutation level $2$ and is isomorphic to the solution on $X = \mathbb{Z}_{p}\times \mathbb{Z}_{p^2}$ given by $$\tilde{\sigma}_{b}(c) = (c_{1} + b_{2}, c_{2} + 1),$$ described in \cite[Theorem 3.1]{JPZ-D21} with $n_1 = p$, $n_2 = p^2$ and $r=0$. Indeed, a bijection $f(a_1, a_2) = (a_{1} + a_2(a_2 -1)/2, a_2)$ gives an isomorphism $f: (X,\tilde{\sigma})\rightarrow (X, \sigma)$ (see \cite[Main Theorem 4.5]{JPZ-D21}) between these solutions.

\subsubsection{Solutions related to \texorpdfstring{$(B_2,+_2,\cdot_2)$}{B2}}
Fix an $a \in \{0,\ldots,p-1\}$ and an $\varepsilon \in \{1\} \cup \F_p\setminus\F_p^2$. Denote by $(B_2,+_2,\cdot_2)$ the brace associated to
\[\lambda_\bb(\x)=\left(x_1+b_2x_2, x_2 + pab_2x_1+px_2\left(\varepsilon b_1+(a-\varepsilon){\binom{b_2}{2}}\right)\right)\]
for all $\bb, \x \in \mathbb{Z}_{p}\times\mathbb{Z}_{p^2}$. 
\begin{lemma}
We have $\langle \Orb((x_1, x_2))\rangle_+ = B_2$ if and only if $x_2$ is an element of order $p^2$ in $\Z_{p^2}$.
\end{lemma}
\begin{proof}
    For an element $\x=(x_1,x_2) \in B_2$, it is clear that if $x_2$ is of order smaller than $p^2$ in $\Z_{p^2}$, then 
    \[\langle\Orb(\x)\rangle_+ \subseteq \langle (1,0),(0,p) \rangle_+ \subsetneq B_2.\]
    On the other hand, if $x_2$ is of order $p^2$ in $\Z_{p^2}$, then 
    \[(1,0) = \lambda_{(b, x^{-1}_2)}(\x) -  \lambda_{(0, 0)}(\x) \in \langle\Orb(\x)\rangle_+,\]
    where
    \[b = \left(- ax_1x^{-1}_2 - (a- \varepsilon )x_2\binom{x^{-1}_2}{2} \right)(\varepsilon x_2)^{-1}.\]
    Moreover, 
    \[(0,1)=x_2^{-1}(\lambda_{(0,0)}(\x)-x_1(1,0)) \in\langle\Orb(\x)\rangle_+.\]
    The assertion follows. 
\end{proof}
    Direct computation shows that
    \[S(\x) = \langle (1,0),(0,p) \rangle \subseteq Z(B_2,\cdot).\]
    Therefore the only core-free subgroup of $S(\x)$ is $\{1\}$, hence all indecomposable solutions with permutation brace $(B_2,+,\cdot)$ have size $p^3$. Now we compute the automorphisms group of $(B_2, +, \cdot)$. 
    \begin{proposition}
    $\Aut(B_2, +, \cdot) $ is isomorphic to 
      \[\left\{ \begin{pmatrix}
          1 & \alpha \\ p(\varepsilon+a)\alpha & p\beta + 1 
      \end{pmatrix}   \middle|  \alpha,\beta \in \Z_p\right\}\cup \left\{ \begin{pmatrix}
          1 & \gamma \\ p((\varepsilon - a)-(\varepsilon+a)\gamma)  & p\delta - 1 
      \end{pmatrix}   \middle|  \gamma,\delta\in \Z_p\right\}\]
      with the multiplication given by matrix multiplication in $\Z$ followed by a reduction modulo $p$ on the first row and a reduction modulo $p^2$ on the second one.  
\end{proposition}
We omit the proof, which is similar to the computation of $\Aut(B_1,+,\cdot)$. 
\begin{proposition}
    There are $(p-1)/2$ non-isomorphic indecomposable solutions with permutation brace isomorphic to $(B_2,+,\cdot)$.
\end{proposition}
\begin{proof}
    Let $\x$ and $\x'$ be such that $x_2$ and $x_2'$ are both elements of order $p^2$ in $\Z_{p^2}$. The solutions associated to $\x$ and $\x'$ are isomorphic if and only if there exists a $\psi \in \Aut(B_2,+,\cdot)$ and a $\z \in B_2$ such that $\psi(\x)=\lambda_\z(\x')$. In particular, it is necessary to have that $x_2 \equiv \pm x_2' \pmod{p}$. We claim that this condition is also sufficient.

    For $\x,\x'$ such that $x_2 \equiv x_2' \pmod{p}$, consider the element
    \[\psi=\begin{pmatrix}
          1 & x_2^{-1}(x_1' - x_1) \\ p(a+\varepsilon)x_2^{-1}(x_1' - x_1) & 1 
      \end{pmatrix} \in \Aut(B_2,+,\cdot).\]
Then, taking $$z_1=(\varepsilon x_2')^{-1}\left( (a+
\varepsilon)x_1x_{2}^{-1}(x_1' - x_1)+ (x_2-x_2')/p\right),$$ we see that $\psi(\x)=\lambda_{(z_1,0)}(\x')$.

For $\x,\x'$ such that $x_2 \equiv -x_2' \pmod{p}$, consider the element
\[\begin{pmatrix}
          1 & x_2^{-1}(x_1' - x_1) \\ -p(a+\varepsilon)x_2^{-1}(x_1' - x_1) + p(\varepsilon - a)  & -1 
      \end{pmatrix} \in \Aut(B_2,+,\cdot).\]
Then, taking
\[z_1=\left(-(a+\varepsilon)x_1x_2^{-1}(x_1' - x_1) + (\varepsilon - a)x_1 - (x_2 + x_2')/p \right)(\varepsilon x_2')^{-1},\]
we see that $\psi(\x)=\lambda_{(z_1,0)}(\x')$.

Given an element $\x \in B_2$, there are in total $2p^2$ elements $\x' \in B_2$ which satisfy $x_2 \equiv \pm x_2' \pmod{p}$. Hence there are $p^2(p-1)/2p^2=(p-1)/2$ non-isomorphic indecomposable solutions with permutation brace $(B_2,+,\cdot)$.
\end{proof}
\begin{solutions}[associated to $B_2$]
    Let $1 \leqslant x \leqslant (p-1)/2$. Then, up to isomorphism, the $(p-1)/2$ indecomposable solutions associated to the brace $(B_2,+,\cdot)$ are given by $(B_2, r)$ where 
    \begin{align*}
  \sigma_\bb(\bc) &= 
    \begin{pmatrix}
      c_1+x(b_2+c_2)\\
      x+c_2+p\left[ac_1x+\varepsilon x(b_1+c_2b_2)+(a-\varepsilon)\left(x{\binom{b_2}{2}}+c_2{\binom{x}{2}}\right)\right]
    \end{pmatrix}, \\
        \tau_{\bc}(\bb) \notag& =
    \begin{pmatrix}
        b_1- (b_2+c_2)x\\
        \begin{matrix}
            b_2-x-p\left[ax(b_1-x(b_2+c_2))+\varepsilon x(b_2(c_2+2x)+c_1-x^2)\right.\\
            \left.+(a-\varepsilon)(x{\binom{x+c_2}{2}}+(b_2-x){\binom{x}{2}})\right]
        \end{matrix}
  \end{pmatrix}.
  \end{align*}
  for all $\bb,\bc \in B_2$.
\end{solutions}   
These solutions are all of multipermutation level $3$. 

\subsubsection{Solutions related to \texorpdfstring{$(B_3,+_3,\cdot_3)$}{B3}} \label{B3}
Fix an $\varepsilon \in \{1\} \cup \F_p\setminus\F_p^2$ and denote by $(B_3,+_2,\cdot_2)$ the brace associated to
\[\lambda_\bb(\x)=\left(x_1+x_2b_2,x_2+p\varepsilon\left(x_1b_2+{\binom{b_2}{2}}x_2\right)\right)\]
for all $\bb, \x \in \mathbb{Z}_{p}\times\mathbb{Z}_{p^2}$. 
\begin{proposition}
    We have $\langle\Orb(\x)\rangle_+=B_3$ if and only if $x_2$ has additive order $p^2$ in $\Z_{p^2}$.
\end{proposition}
\begin{proof}
    It is clear that, if $x_2$ does not have order $p^2$, then $(0,1) \notin \langle \Orb(\x)\rangle_+$. Conversely, take $b_1=(0,-x_1x_2^{-1})$ and
    \[m=\left(x_2+p\varepsilon\left(x_1b_1+{\binom{b_1}{2}}x_2\right)\right)^{-1},\]
    then
    \[(0,1)=m\lambda_{(0,b_2)}(\x) \in \langle\Orb(\x)\rangle_+.\]
    Moreover, consider the element $(y_1,y_2)=\lambda_{(0,b_2)}(\x)$ where $b_2=(1-x_1)x_2^{-1}$. Then $y_1=1$, and hence
    \[(1,0)=(y_1,y_2)-y_2(0,1) \in \langle\Orb(\x)\rangle_+.\]
\end{proof}
It follows that there are $p^2(p-1)$ choices for $\x$. Given such an $\x$, therefore, it is easy to see that
\[S(\x)=\{(b_1,b_2) \in B   \mid b_2 \equiv 0 \pmod{p} \} \cong \Z_p \times \Z_p.\]
By Remark~\ref{rem:size-corefree}, to compute the (non-trivial) core-free subgroups $K$ of $(B_3,\cdot)$ contained in $S(\x)$, it suffices to ask which subgroups of $S(\x)$ of order $p$ have trivial core. The subgroups of $S(\x)$ of order $p$ are of the form $\langle (i,p) \rangle$ for $i \in \{0,\ldots,p-1\}$ or $\langle (1,0) \rangle$. We now need the following lemma.

    \begin{lemma}\label{lem:core-free}
      Let $G$ be a non-abelian group of order $p^3$. Then every subgroup of order $p$ which is not central in $G$ is core-free.
    \end{lemma}
    \begin{proof}
      For any subgroup $H$ of order $p$ we know that $H$ is either core-free or normal in $G$. Moreover, as $G$ is a finite $p$-group, for any non-trivial normal subgroup $N$ we know that $|Z(G)\cap N| > 1$. If, therefore, $H$ is normal in $G$, it follows that $H=Z(G)$ because $|H|=|Z(G)|=p$. Otherwise $H$ must be core-free with $H \cap Z(G)=\{1\}$.
    \end{proof}

    Note that $\x\in Z(B_3,\cdot)$ if and only if for arbitrary $\mathbf{y}\in B_3$ we have that
    \[x_2y_1 + {\binom{x_2}{2}} y_2 - x_1y_2 - {\binom{y_2}{2}} x_2 \equiv 0 \pmod{p}.\]
    In particular, $\{ (0, pk)   \mid k= 0, \ldots, p-1\}\subseteq Z(B_3,\cdot)$. As $|Z(B_3,\cdot)| = p$ we get that in fact $\{ (0, pk)   \mid k= 0, \ldots, p-1\} = Z(B_3,\cdot)$. It follows from Lemma \ref{lem:core-free} that every subgroup of order $p$ different from $\langle (0, p)\rangle$ is core-free. For the rest of this section, we will denote the subgroup $\langle(1,0)\rangle$ by $K_0$ and the subgroup $\langle(i,p)\rangle$ by $K_i$ for each $1 \leqslant i \leqslant p-1$. We remark in particular that any indecomposable solution associated with $K_i$, for $0 \leqslant i \leqslant p-1$, will have size $p^2$.
    \begin{proposition}
    $\Aut(B_3, +, \cdot) $ is isomorphic to 
      \[\left\{ \begin{pmatrix}1 & \alpha \\ p\varepsilon \alpha & p\beta + 1 \end{pmatrix}   \middle|  \alpha,\beta\in \Z_p\right\}
        \cup 
        \left\{ \begin{pmatrix}1 & \gamma \\ -p(\varepsilon \gamma +1) & p\delta - 1 \end{pmatrix}   \middle|  \gamma,\delta\in \Z_p\right\}\]
      with the multiplication given by matrix multiplication in $\Z$ followed by a reduction modulo $p$ on the first row and a reduction modulo $p^2$ on the second one.  
    \end{proposition}
    We again omit the computational proof, and move onto understanding when two solutions are isomorphic. For ease of notation, let us denote an element of $\Aut(B_3,+,\cdot)$ of the form
    \[\begin{pmatrix}1 & \alpha \\ p\varepsilon \alpha & p\beta + 1 \end{pmatrix}\]
    by $\psi_1$, and an element of $\Aut(B_3,+,\cdot)$ of the form
    \[\begin{pmatrix}1 & \gamma \\ -p(\varepsilon \gamma +1) & p\delta - 1 \end{pmatrix}\]
    by $\psi_2$. 
    
  \begin{proposition}
      Let $\x,\x'\in B_3$ be such that $x_2,x_2'$ both have order $p^2$ in $\Z_{p^2}$. Then
      \[(\x,\{1\}) \sim (\x',\{1\})\]
      if and only if $x_2 \equiv \pm x_2' \pmod{p}$.
  \end{proposition}
  \begin{proof}
      We have that $(\x,\{1\}) \sim (\x',\{1\})$ if and only if there is a $\z \in B_3$ and either a $\psi_1$ or $\psi_2$ in $\Aut(B_3,+,\cdot)$ such that
    \begin{equation}\label{isom1}
        \psi_i(\x) = \lambda_{z}(\x'),
    \end{equation}
    where $i=1$ or $i=2$. Suppose first that $i=1$. Condition (\ref{isom1}) is equivalent to the fact that there is a $\z \in B_3$ and an $\alpha,\beta \in \Z_p$ such that
        \begin{align}
          x_1 + x_2 \alpha \equiv& x_1' +x_2'z_2 \pmod{p} \label{isom1-1}\\
          x_2 + p\beta x_2 + p\varepsilon x_1\alpha \equiv& x_2' + p\varepsilon \left( x_1' z_2 + { \binom{z_2}{2}}x_2'\right) \pmod{p^2} \label{isom1-2}
      \end{align}
    From \eqref{isom1-2}, it is clear that it is necessary to have $x_2\equiv x_2' \m p$. Now, let $\eta$ be such that $p\eta = x_2' - x_2$, then taking $\alpha = x_2^{-1}(x_1' - x_1)$ and $\beta = (\eta -\varepsilon x_1x_2^{-1}(x_1' - x_1))x_2^{-1}$, we have $\psi_1(\x)=\lambda_{(0,0)}(\x')$. Hence $(\x,\{1\}) \sim (\x',\{1\})$ whenever $x_2 \equiv x_2' \pmod{p}$.

    Similarly, if $i=2$, we get that (\ref{isom1}) holds if and only if $x_2 \equiv -x_2' \pmod{p}$.
  \end{proof}
  We therefore obtain $(p-1)/2$ non-isomorphic indecomposable solutions of size $p^3$ with permutation brace $(B_3,+,\cdot)$.

  We now consider the solutions generated by the pairs $(\x,K)$ where $K=K_i$ for some $0 \leqslant i \leqslant p-1$.
  \begin{proposition}
      Let $\x$ be such that $x_2$ has order $p^2$ in $\Z_{p^2}$. Then
      \[(\x,K_i) \not\sim (\x,K_j)\]
      for any $i \neq j$. In particular, there are $p(p-1)/2$ non-isomorphic indecomposable solutions of size $p^2$ with permutation brace $(B_3,+,\cdot)$. 
  \end{proposition}
  \begin{proof}
  By Remark \ref{elim-elts}, we know that if $(\x,\{1\}) \sim (\x',\{1\})$, then for any core-free $K \leqslant S(\x)$, there is a core-free $K' \leqslant S(\x')$ such that $(\x,K) \sim (\x',K')$. Thus our problem restricts to determining the $K_i,K_j \leqslant S(\x)$ for which $(\x,K_i) \sim (\x,K_j)$.

  Firstly, we must consider the pairs $(\z,\psi)$ such that $\psi(\x)=\lambda_\z(\x)$. By the previous proof, this is equivalent to
  \begin{align*}
      \alpha &\equiv z_2 \pmod{p},\\
      \beta  &\equiv \varepsilon \binom{z_2}{2}x_2^{-1} \pmod{p},
  \end{align*}
  where $\psi=\psi_1$. Thus a unique suitable pair $(\z,\psi)$ exists for any $\z \in B$. For a given suitable pair $(\z,\psi_1)$ such that $p \nmid z_2$ and for $0 \leqslant i \leqslant p-1$, we see that
  \[\psi_1(K_i)=\begin{cases}
      K_{\varepsilon^{-1} z_2^{-1}}, &i=0\\
      K_{i(\varepsilon z_2i +1)}, &i>0.
  \end{cases}\]
  Finally, computing $z^{-1}\psi_1(K_i)z$ in each case, we see that
  \[z^{-1}\psi_1(K_i)z=K_i.\]
  In particular, $\psi(K_i)=zK_iz^{-1}$ for every pair $(\psi,z)$ such that $\psi(x)=\lambda_z(x)$, and hence we can never have $(\x,K_i) \sim (\x,K_j)$ for $i \neq j$. For each $1 \leqslant x \leqslant (p-1)/2$, therefore, the pair $((0,x),K)$ such that $K \leqslant S((0,1))$ gives an indecomposable solution of size $p^2$ with permutation brace isomorphic to $(B_3,+,\cdot)$, with no two such pairs generating isomorphic solutions.
  \end{proof}

\begin{solutions}[associated to $B_3$]
    Let $1 \leqslant x \leqslant (p-1)/2$ and
    \[K \in \{\{1\},\langle(1,0)\rangle,\langle(i,p)\rangle \mid 1 \leqslant i \leqslant p-1\}.\]
    Then, up to isomorphism, the $(p^2-1)/2$ indecomposable solutions associated to the brace $(B_3,+,\cdot)$ are given by $(B_3/K, r)$, where
        \begin{align*}
  \sigma_{\bb \cdot K }(\bc \cdot K) & =
  \begin{pmatrix}
      c_1 +x(b_2+c_2)\\
      c_2+x+p\varepsilon\left({\binom{x}{2}}c_2+x{\binom{b_2}{2}}+xc_1\right)
  \end{pmatrix}\cdot K, \\
  \tau_{\bc \cdot K}(\bb \cdot K) & =
  \begin{pmatrix}
      b_1 -x (b_2 + c_2)\\
      b_2-x+p\varepsilon\left(x^2c_2+x^3-\binom{c_2 + x}{2} x+\binom{x+1}{2}(b_2-x)-b_1x\right)
  \end{pmatrix}\cdot K,
  \end{align*}
  for all $\bb,\bc \in B_3$.
\end{solutions}

For $K=\{1\}$ the obtained solution is of multipermutation level $3$, otherwise the solutions are of multipermutation level $2$. 

Let us rewrite the obtained solutions of size $p^2$ to the form without cosets. In the case of non-trivial subgroup $K_{i}$, where $i\in\{0, \ldots, p-1\}$ it is straightforward to check that $\{\binom{0}{b} \mid b\in \mathbb{Z}_{p^2}\}$ is left transversal of $K_{i}$ in $(B, \cdot)$. Then, after identifying $\binom{b_1}{b_2}\cdot K_{i}$ with $\binom{0}{c}\cdot K_i$ for uniquely determined $c\in \mathbb{Z}_{p^2}$, the solution can be rewritten in the following way. 

$$ \sigma_{b}(c) = x + c + p\varepsilon \left( \binom{x}{2} c + x \binom{b}{2}  - x(x+c)(b+c)\right) - pjx(b+c),$$
where $j$ is multiplicative inverse of $i$ in $\mathbb{Z}_{p}$ for $i\neq 0$ and $j=0$, when $i = 0$.

\begin{corollary}
Up to isomorphism there are $ 1 + (2p+1)(p^2-1)/4$ indecomposable solutions with permutation brace of type $\Z_p \times \Z_{p^2}$, among them $ 1 + (p+1)(p^2-1)/4 $ are of size $p^3$.  
\end{corollary}

\subsection{Braces of elementary abelian type}
In this section, we investigate the indecomposable solutions admitting a permutation brace of type $\Z_p^3$. In what follows, let $0 \leqslant \mu \leqslant (p-1)/2$, $0 \leqslant a \leqslant p-1$ and $\varepsilon \in \{1\} \cup \F_p\setminus\F_p^2$. We will denote an arbitrary element $\x \in \Z_p^3$ by $(x_1,x_2,x_3)$.

By \cite{Bac15}, there are $(p^2 + 6p + 21)/4$ braces of elementary abelian type up to isomorphism. This first proposition shows that at least $(p^2 + 2p + 17)/4$ of these cannot be the permutation brace of an indecomposable solution:
\begin{proposition}\label{prop:no-ind-sols2}
    Let $(B,+,\cdot)$ be a brace with one of the following $\lambda$-maps for all $\bb, 
\x\in \mathbb{Z}_{p}^3$:
    \begin{enumerate}
        \item $\lambda_\bb(\x)=(x_1-b_3x_2+b_2x_3,x_2,x_3)$,
        \item $\lambda_\bb(\x)=(x_1+(\varepsilon b_2+\mu b_3)x_2+b_3x_3,x_2,x_3)$,
        \item $\lambda_\bb(\x)=(x_1+b_2x_2,x_2,x_3)$,
        \item $\lambda_\bb(\x)=(x_1+b_3x_2,x_2,x_3)$,
        \item $\lambda_\bb(\x)=\x$.
    \end{enumerate}
    Then there are no indecomposable solutions with $(B,+,\cdot)$ as a permutation brace.
\end{proposition}
\begin{proof}
    Brief computation shows that for every element $\x$ in any of the braces above, we have
    \[\langle \Orb(\x)\rangle_+\subseteq\langle(a, x_2, x_3) \mid a\in \Z_p\rangle_+ = \{ (a, kx_2, kx_3) \mid a \in \Z_p, k\in\Z\}.\]
    It is clear that this is a proper subset of $B$, so we cannot satisfy the requirements of Theorem \ref{BCJ}
\end{proof}
There are two (families of) braces of elementary abelian type which are not included in the list of Proposition \ref{prop:no-ind-sols2}. These are given by the following $\lambda$-maps for all $\bb, \x\in \mathbb{Z}_{p}^3$:
\begin{enumerate}
    \item $\lambda_\bb(\x)=\left(x_1+ab_3x_2+b_2x_3+(a-1){\binom{b_3}{2}}x_3,x_2+b_3x_3,x_3\right)$,
    \item $\lambda_\bb(\x)=\left(x_1+b_3x_2+{\binom{b_3}{2}}x_3,x_2+b_3x_3,x_3\right)$.
\end{enumerate}
We will denote these by $(C_1,+_1,\cdot_1)$ and $(C_2,+_2,\cdot_2)$ respectively, often abusing notation by ignoring the indices of the brace operations. We will show that they each lead to indecomposable solutions.

\subsubsection{Solutions related to \texorpdfstring{$(C_1,+_1,\cdot_1)$}{C1}}
Fix an $a \in \{0,\ldots,p-1\}$ and denote by $(B_2,+_2,\cdot_2)$ the brace associated to
\[\lambda_\bb(\x)=\left(x_1+ab_3x_2+b_2x_3+(a-1){\binom{b_3}{2}}x_3,x_2+b_3x_3,x_3\right)\]
for all $\bb, \x\in \mathbb{Z}_{p}^3$.
\begin{proposition}
    Let $\x \in C_1$, then we have $\langle\Orb(\x)\rangle_+=C_1$ if and only if $x_3 \neq 0$.
\end{proposition}
\begin{proof}
    It is clear that if $x_3=0$, we have $(0,0,1) \notin \langle\Orb(\x)\rangle_+$, so assume from now on that $x_3 \neq 0$. Firstly, we see that
    \[(1,0,0)=\lambda_{(0,x_3^{-1},0)}(\x)-\lambda_{(0,0,0)}(\x) \in \langle\Orb(\x)\rangle_+.\]
    Next, note that the element $(y_1,y_2,y_3)=\lambda_{(0,0,x_3^{-1})}(\x)-\lambda_{(0,0,0)}(\x)$ is such that $y_2=1$ and $y_3=0$, and therefore we have
    \[(0,1,0)=(y_1,y_2,y_3)-y_1(1,0,0) \in \langle \Orb(\x)\rangle_+.\]
    Finally, we see that
    \[(0,0,1)=x_3^{-1}(\lambda_{(0,0,0)}(\x)-x_1(1,0,0)-x_2(0,1,0))\in \langle \Orb(\x)\rangle_+.\]    
\end{proof}
    Direct computation shows that for every $\x\in C_1$ such that $x_3\neq 0$ we have
    \[S(\x) = \{ (x, 0, 0)   \mid x\in \Z_p\}.\]
    It is straightforward to check that $ \{ (x, 0, 0)   \mid x\in \Z_p\} \subseteq Z(C_1,\cdot)$ and thus the only core-free subgroup of this group is trivial. Therefore all indecomposable solutions with permutation brace $(C_1,+,\cdot)$ have size $p^3$. We now describe $\Aut(C_1,+,\cdot)$. 
 \begin{proposition}
    $\Aut(C_1, +, \cdot)$ is isomorphic to 
    \[\left\{\begin{pmatrix}
          \alpha^3 & \delta & \gamma \\
          0 & \alpha^2 & \beta \\ 
          0 & 0 & \alpha
\end{pmatrix}   \middle|  \; \alpha,\beta,\gamma \in \Z_p, \alpha \neq 0, \delta = \alpha (a+1) \beta + (a-1)\alpha^2(\alpha - 1)/2\right\}
.\] 
 \end{proposition}
 \begin{proof}
     It is clear that $\Aut(C_1,+)=GL_3(\Z_p)$. Let us therefore consider the element
     \[A=\begin{pmatrix} a_{11} & a_{12} & a_{13} \\
                    a_{21} & a_{22} & a_{23} \\
                    a_{31} & a_{32} & a_{33}
\end{pmatrix} \in \Aut(C_1,+,\cdot).\]
We require $A$ to satisfy
\begin{equation}\label{brace_hom}
    A(\x\cdot\x')=A(\x)\cdot A(\x')
\end{equation}
for all $\x,\x' \in C_1$. For the meantime, let us set $\alpha=a_{33}, \beta=a_{23}, \gamma=a_{13}$, and $\delta = a_{12}$, and write $f=f(x_2,x_3,x_2',x_3')$ for the following expression
\[ax_3x_2'+x_2x_3'+(a-1)\binom{x_3}{2}x_3'.\]
Firstly, consider taking $\x=(0,0,x_3)$ and $\x'=(0,x_2',x_3')$. Then the second and third coordinates of (\ref{brace_hom}) give
\begin{align}
    a_{21}f(0,x_3,x_2',x_3')+a_{22}x_3x_3'&=\alpha x_3(a_{32}x_2'+\alpha x_3').\label{m1}\\
    a_{31}f(0,x_3,x_2',x_3')+a_{32}x_3x_3'&=0.\label{m2}
\end{align}
From \eqref{m2} we easily obtain that $a_{31}=a_{32}=0$. Substituting this into \eqref{m1}, we likewise get $a_{21}=0$ and $a_{22}=\alpha^2$. Conversely, if these conditions are satisfied, \eqref{m1} and \eqref{m2} hold. With these values determined and now taking $\x=(0,x_2,x_3)$ and $\x'=(0,x_2',x_3')$, the first coordinate of (\ref{brace_hom}) gives
\begin{equation}\label{m3}
    a_{11}f+\delta x_3x_3'=a\alpha x_3(\alpha^2x_2'+\beta x_3')+\alpha(\alpha^2x_2+\beta x_3)x_3'+\alpha(a-1)\binom{\alpha x_3}{2}x_3'.
\end{equation}
Substituting $(x_2,x_3)=(x_2',x_3')=(0,1)$ into (\ref{m3}) gives
\[\delta=\alpha\beta(a+1)+\alpha(a-1)\binom{\alpha}{2}.\]
With $\delta$ now determined, substituting it back into (\ref{m3}) with $x_3=2$, $x_2'=0$ and $x_3'=1$ gives
\[a_{11}(x_2+(a-1))=\alpha^3x_2+(a-1)\alpha^3.\]
By comparing coefficients, it is now easy to see that $a_{11}=\alpha^3$. Finally, it can be checked that the matrix
\[A=\begin{pmatrix}
          \alpha^3 & \delta & \gamma \\
          0 & \alpha^2 & \beta \\ 
          0 & 0 & \alpha
\end{pmatrix}\]
is indeed an element of $\Aut(C_1,+,\cdot)$.
 \end{proof}

We will denote an element of $\Aut(C_1,+,\cdot)$ with given $\alpha,\beta$ and $\gamma$ by $\psi_{\alpha,\beta,\gamma}$.
\begin{proposition}
    Up to isomorphism, there is a unique indecomposable solution with permutation brace $(C_1,+,\cdot)$ for a given $a\in\{0, \ldots, p-1\}$.
\end{proposition}
\begin{proof}
    For elements $\x,\x' \in C_1$ such that $x_3,x_3' \neq 0$, consider the automorphism $\psi_{\alpha,\beta,\gamma}$ with $\alpha=x_3'x_3^{-1}$, $\beta=(x_2'-\alpha^2x_2)x_3^{-1}$ and $\gamma=0$, and set $z_2=(\alpha^3x_2+\delta x_2-x_1')x_3^{-1}$. Then we see that
    \[\psi_{\alpha,\beta,\gamma}(\x)=\lambda_{(0,z_2,0)}(\x').\]
    In particular, all indecomposable solutions with permutation brace $(C_1,+,\cdot)$ are isomorphic.
\end{proof}
By the above, we may simply consider the solution associated to the element $\x=(0,0,1)$, which leads us to the solution as follows. 

\begin{solution}[associated to $C_1$]
Up to isomorphism, there is a unique solution associated to a brace $(C_1, +, \cdot)$, which is given by $(C_1, r)$, where

   \begin{align*}
  \sigma_\bb(\bc) & =
  \begin{pmatrix}
  b_2+ c_1+ b_3c_3+ ac_2 +(a-1){\binom{b_3}{2}} \\
  b_3 + c_2 + c_3\\
  c_3+1
  \end{pmatrix},\\
  \tau_{\bc}(\bb) & =
  \begin{pmatrix}
   b_1 + c_2+ 2c_3- b_3c_3 +a(b_3-b_2+c_3) + (a-1)\binom{c_3+ 1}{2}\\ 
    b_2- b_3- c_3\\
    b_3 - 1
  \end{pmatrix},
  \end{align*}
  for all $\bb,\bc \in C_1$.
\end{solution}
This solution is of multipermutation level $3$.

\subsubsection{Solutions related to \texorpdfstring{$(C_2,+_2,\cdot_2)$}{C2}}
Denote by $(C_2,+_2,\cdot_2)$ the brace associated~to
\[\lambda_\bb(\x)=\left(x_1+b_3x_2+{\binom{b_3}{2}}x_3,x_2+b_3x_3,x_3\right)\]
for all $\bb, \x\in \mathbb{Z}_{p}^3$.
\begin{proposition}
    Let $\x \in C_2$, then we have $\langle \Orb(\x)\rangle_+=C_2$ if and only if $x_3 \neq 0$.
\end{proposition}
\begin{proof}
    It is clear that if $x_3=0$, then $(0,0,1) \notin \langle\Orb(\x)\rangle_+$, so assume from now on that $x_3 \neq 0$. Firstly, we see that
    \[(1,0,0)=x_3^{-1}\left(\lambda_{(0,0,0)}(\x)+\lambda_{(0,0,2)}(\x)-2\lambda_{(0,0,1)}(\x)\right) \in\langle\Orb(\x)\rangle_+.\]
    Secondly, we see that
    \[(x_2,x_3,0)=\lambda_{(0,0,1)}(\x)-\lambda_{(0,0,0)}(\x),\]
    and hence
    \[(0,1,0)=x_3^{-1}((x_2,x_3,0)-x_2(1,0,0)) \in\langle\Orb(\x)\rangle_+.\ \]
    Finally, we see that
    \[(0,0,1)=x_3^{-1}\left(\lambda_{(0,0,0)}(\x)-x_1(1,0,0)-x_2(0,1,0)\right)\in\langle\Orb(\x)\rangle_+.\]
\end{proof}
Direct computation shows that for every $\x \in C_2$ such that $x_3 \neq 0$ we have
\[S(\x) = \{ (b_1, b_2, 0)   \mid b_1, b_2 \in \Z_p\} \cong \Z_p \times \Z_p.\]
By Lemma~\ref{lem:core-free}, to compute the (non-trivial) core-free subgroups $K$ of $(C_2,\cdot)$ contained in $S(\x)$, it suffices to ask which subgroups of $S(\x)$ of order $p$ have trivial core. The subgroups of $S(\x)$ of order $p$ are of the form $\langle(i,1,0)\rangle$ for $i \in \{0,\ldots,p-1\}$ or $\langle(1,0,0)\rangle$. It is straightforward to compute that $Z(C_2,\cdot)=\langle(1,0,0)\rangle$, and hence by Lemma \ref{lem:core-free}, as $(C_2,\cdot)$ is non-abelian, the non-trivial core-free subgroups of $S(\x)$ are precisely $\langle(i,1,0)\rangle$ for $0 \leqslant i \leqslant p-1$. For a fixed $i$, we will denote the corresponding core-free subgroup by $K_i$. Any indecomposable solution associated with $K_i$ for $0 \leqslant i \leqslant p-1$ will have size $p^2$.
 \begin{lemma}
    $\Aut(C_2, +, \cdot)$ is isomorphic to 
    \[\left\{\begin{pmatrix}
          \alpha^3 & \delta & \gamma \\
          0 & \alpha^2 & \beta \\ 
          0 & 0 & \alpha
\end{pmatrix}   \middle|  \alpha, \beta, \gamma \in \Z_p, \alpha \neq 0, \delta = \alpha\beta + \alpha^2(\alpha -1)/2 \right\}
.\]
 \end{lemma}
 We will denote an element of $\Aut(C_2,+,\cdot)$ with given $\alpha,\beta$ and $\gamma$ by $\psi_{\alpha,\beta,\gamma}$.
\begin{proof}
We have that $\Aut(C_2, +) = GL_3(\Z_p)$. Let us therefore consider the element
\[A=\begin{pmatrix} a_{11} & a_{12} & a_{13} \\
                    a_{21} & a_{22} & a_{23} \\
                    a_{31} & a_{32} & a_{33}
\end{pmatrix} \in \Aut(C_2,+,\cdot).\]
Firstly, taking $\x=(0,0,1)$ and $\x'=(0,1,0)$, then asking for $A(\x\cdot\x')=A(\x)\cdot A(\x')$ tells us that:
\begin{align}
    a_{11}&=a_{33}a_{22}+\binom{a_{33}}{2}a_{32}, \label{mat1}\\
    a_{21}&=a_{33}a_{32} \label{mat2},\\
    a_{31}&=0 \label{mat3}.
\end{align}
Now taking $\x=\x'=(0,0,1)$ tells us that
\begin{align}
    a_{32}&=0, \label{mat4}\\
    a_{22}&=a_{33}^2. \label{mat5}
\end{align}
Hence by (\ref{mat2}) and (\ref{mat4}), we have $a_{21}=0$, and by (\ref{mat1}) and (\ref{mat4}), we have $a_{11}=a_{33}a_{22}$. Denoting $a_{33}$ by $\alpha$, we have $a_{22}=\alpha^2$ and $a_{11}=\alpha^3$. Note that $\alpha \neq 0$, otherwise $A$ does not describe an element of $GL_3(\Z_p)$. We note also that the second and third components of $A(\x\cdot\x')=A(\x)\cdot A(\x')$ have now been satisfied, and we are left with the following equation:
\[\alpha^3x_3'{\binom{x_3}{2}}+a_{12}x_3x_3'=\alpha a_{23}x_3x_3'+\alpha x_3'\binom{\alpha x_3}{2}.\]
If either $x_3=0$ or $x_3'=0$, then this equation is satisfied. So we may consider the equation for $x_3,x_3' \neq 0$, which then simplifies to
\[a_{12}+\alpha^2/2=\alpha a_{23}+\alpha^3/2.\]
Denoting $a_{23}$ by $\beta \in \Z_p$, we get
$a_{12}=\alpha\beta+\alpha^2(\alpha-1)/2$. Finally, we note that $A$ is invertible if and only if $\alpha \neq 0$, so there are no further restrictions on $\alpha,\beta$ or $a_{31}$ (which we will denote by $\gamma$). The assertion now follows.
\end{proof}
For an element $\x \in C_2$ with $x_3 \neq 0$ and a core-free subgroup $K$ of $S(\x)$, we again consider the pairs $(\x,K)$, and their associated (indecomposable) solutions as in Section \ref{B3}. We will also denote by $A_{\alpha,\beta,\gamma}$ an element of $\Aut(C_2,+,\cdot)$ with given $\alpha,\beta$ and $\gamma$.
\begin{lemma}
    For any $\x,\x' \in C_2$ such that $x_3,x_3' \neq 0$, we have
    \[(\x,\{1\}) \sim (\x',\{1\}).\]
\end{lemma}
\begin{proof}
    Consider $A_{\alpha,\beta,\gamma} \in \Aut(C_2,+,\cdot)$ such that
    \begin{align*}
        \alpha &= x_3x_3'^{-1},\\
        \beta &= (x_2-\alpha x_2')x_3'^{-1},\\
        \gamma &= (x_1-\alpha^3 x_1-\delta x_2')x_3'^{-1}.
    \end{align*}
    Then \[A_{\alpha,\beta,\gamma}(\x)=\lambda_{(0,0,0)}(\x').\]
    The assertion follows.    
\end{proof}
Let us now consider solutions of size $p^2$. Recall that for $i \in \Z_p$, we denote by $K_i$ the group $\langle (i,1,0)\rangle \leqslant S(\x)$ (for any $\x \in C_2$ as above).

\begin{proposition}
    For any $\x\in C_2$ such that $x_3 \neq 0$ and for $0 \leqslant i,j \leqslant p-1$, we have
    \[(\x,K_i) \sim (\x,K_j)\]
    if and only if $i=j$.
\end{proposition}
\begin{proof} For $\z \in C_2$, it is checked that
\[\z K_i \z^{-1}=K_{i+z_3},\]
and that
\[A_{\alpha,\beta,\gamma}(K_i)=K_{\alpha i+\alpha^{-2}\delta}\]
for any $A_{\alpha,\beta,\gamma} \in \Aut(C_2,+,\cdot)$. We have that $(\x,K_i) \sim (\x,K_j)$ if and only if there is an $A_{\alpha,\beta,\gamma}\in \Aut(C_2,+,\cdot)$ and a $\z \in C_2$ such that $A_{\alpha,\beta,\gamma}(\x)=\lambda_\z(\x)$ and $A_{\alpha,\beta,\gamma}(K_i)=\z K_j\z^{-1}$. From $A_{\alpha,\beta,\gamma}(\x)=\lambda_\z(\x)$ and the fact that $x_3\neq 0$ it follows easily that $\alpha = 1$, $\beta = z_3$, $\gamma = \binom{z_3}{2} $ and $\delta = z_3$. Then $A_{\alpha,\beta,\gamma}(K_{i}) = K_{i + z_3}$ and thus $A_{\alpha,\beta,\gamma}(K_{i}) = z K_{j}z^{-1}$ gives $i = j$, as postulated.
\end{proof}

One may obtain all non-isomorphic solutions by taking
$x = (0, 0, 1)$ and $K\in\{K_{i}\quad |\quad i= 0, \ldots, p-1\}\cup\{\{1 \}\}$. Namely, we get the following.

\begin{solutions}[associated to $C_2$] Let 
\[K\in\{\{1\},\langle(i, 1,0)\rangle \quad | \quad i=0, \ldots, p-1\}.\]
Then, up to isomorphism, the $p+1$ indecomposable solutions associated to $(C_2, +, \cdot)$ are given by $(C_{2}/K, r)$, where 
\begin{align*}
  \sigma_{\bb \cdot K }(\bc \cdot K) & =
  \begin{pmatrix}
    {\binom{b_3}{2}} + c_1 + c_2 \\
    b_3 + c_2 + c_3\\
    c_3 +1
  \end{pmatrix}\cdot K,\\
  \tau_{\bc \cdot K}(\bb \cdot K) & =
     \begin{pmatrix}
     b_1 - b_2 + \binom{c_3+1}{2} - 1\\
      b_2 - b_3 - c_3 \\
      b_3 - 1
    \end{pmatrix}
    \cdot K,
  \end{align*}
  for all $\bb,\bc \in C_2$.
\end{solutions}
All of these solutions are of multipermutation level 2.


We write the obtained solutions in the form without cosets. Note that for non-trivial $K_{i}$, the set $$ 
\left\{ \begin{pmatrix} b_1 \\ 0 \\ b_2 \end{pmatrix} \middle| (b_1, b_2)\in \mathbb{Z}_{p}\times \mathbb{Z}_{p}\right\}$$ is transversal of $K_i$. Therefore, the above solution can be rewritten in the following way. Let $X= \Z_p\times\Z_p$. Then for arbitrary $b = (b_1, b_2), c = (c_1, c_2)$ we have 
\[\sigma_b(c) = \begin{pmatrix}
 \binom{b_2}{2} + c_1 - (i+1 +c_2)(b_{2}+ c_2) \\ c_2 + 1 
\end{pmatrix}.\]

\begin{corollary}
There are $2p + 1$ indecomposable solutions with permutation brace of type $\Z_p^3$, $p+1$ of them are of size $p^3$.
\end{corollary}




\section{A computer algorithm for constructing indecomposable solutions}\label{alg}
The {\sc Gap} package \texttt{YangBaxter}, \cite{YangBaxter} contains a database of (skew) braces up to size 168, computed with the use of the algorithm given in \cite{GV17}. This makes it possible to write a script which encodes the construction of Bachiller, Ced\'o and Jespers, to produce a complete list of (non-degenerate involutive) indecomposable solutions to \eqref{YBE} whose permutation braces have small order. To our knowledge, this is the first systematic construction and enumeration of indecomposable solutions of small size.

The following algorithm makes use of the observation in Remark \ref{elim-elts} to circumvent the need for checking whether two computed solutions are isomorphic, and vastly reduces the elements that need to be checked as potential generators of solutions. We note, however, that the computation of the automorphism group of a given brace $(B,+,\cdot)$ appears to be a significant bottleneck. The way this is typically done is to compute the automorphism group $A$ of either $(B,+)$ or $(B,\cdot)$ (the latter of the two in our case) and then checking which elements of $A$ preserve the structure of the other group. Unfortunately, this turns out to be incredibly time consuming, and computing the automorphism groups for all braces of a given order, say greater than 90, can take several days.

\subsection{The algorithm}
Let $n,k$ be positive integers and let $(B,+,\cdot)$ be the $k^\text{th}$ brace of order $n$. We proceed in {\sc Gap} as follows:
\begin{enumerate}[label=Step \arabic*.]
    \item Create the set $X=B$.
    \item Select a random element $x \in X$. If $|\langle \Orb(x) \rangle_+|=n$, then add $x$ to the set $\mathrm{UniGens}$. Create the set
    \[\mathrm{EquivSols}_x=\{(\varphi(\lambda_b(x)),\varphi,\lambda_b(x)) \mid \varphi \in \Aut(B,+,\cdot), b \in B\},\]
    and redefine $X=X \setminus \{a \mid (a,b,c) \in \mathrm{EquivSols}_x\}$. If $X=\emptyset$, move on to Step 3, otherwise repeat Step 2.
    \item For each $x \in \mathrm{UniGens}$, compute the stabiliser $S(x)$ of $x$ with respect to the $\lambda$-action of $(B,\cdot)$ and create the set
    \[Y_x=\{K \leqslant S(x) \mid K \text{ is core-free}\}.\]
    \item For each $x \in \mathrm{UniGens}$, do the following:
        \begin{enumerate}[label=\alph*)]
            \item Select a random element $K_x \in Y_x$ and add the pair $(x,K_x)$ to the set $\mathrm{Sols}_x$.
            \item Redefine
    \[Y_x=Y_x \setminus \{\varphi(bK_xb^{-1}) \mid (x,\varphi,\lambda_b(x)) \in \mathrm{EquivSols}_x\}.\]
            \item If $Y_x \neq \emptyset$ then repeat the procedure from a).
        \end{enumerate}
    \item We then have a set of representatives
    \[\{\{(x,K_x) \mid K_x \in \mathrm{Sols}_x\} \mid x \in \mathrm{UniGens\}},\]
    where each pair $(x,K_x)$ can be used to generate an indecomposable solution, and any two pairs generate non-isomorphic solutions.
\end{enumerate}

\begin{table}
    \centering
    \begin{tabular}{|c|c|}
    \hline
    $n$ & \#IndSols\\
    \hline
2 & 1\\
3 & 1\\
4 & 3\\
5 & 1\\
6 & 2\\
7 & 1\\
8 & 13\\
9 & 4\\
10 & 2\\
11 & 1\\
12 & 11\\
13 & 1\\
14 & 2\\
15 & 1\\
16 & 86\\
17 & 1\\
18 & 9\\
19 & 1\\
20 & 10\\
21 & 3\\
22 & 2\\
23 & 1\\
24 & 62\\
25 & 6\\
26 & 2\\
27 & 27\\
28 & 7\\
29 & 1\\
30 & 4\\
31 & 1\\
32 & ?\\
33 & 1\\
34 & 2\\
35 & 1\\
36 & 67\\
37 & 1\\
\hline
\end{tabular}
\begin{tabular}{|c|c|}
\hline
$n$ & \#IndSols\\
\hline
38 & 2\\
39 & 3\\
40 & 52\\
41 & 1\\
42 & 8\\
43 & 1\\
44 & 7\\
45 & 4\\
46 & 2\\
47 & 1\\
48 & 434\\
49 & 8\\
50 & 13\\
51 & 1\\
52 & 10\\
53 & 1\\
54 & 98\\
55 & 5\\
56 & 69\\
57 & 3\\
58 & 2\\
59 & 1\\
60 & 28\\
61 & 1\\
62 & 2\\
63 & 14\\
64 & ?\\
65 & 1\\
66 & 4\\
67 & 1\\
68 & 10\\
69 & 1\\
70 & 4\\
71 & 1\\
72 & 488\\
73 & 1\\
\hline
    \end{tabular}
    \begin{tabular}{|c|c|}
\hline
$n$ & \#IndSols\\
\hline
74 & 2\\
75 & 13\\
76 & 7\\
77 & 1\\
78 & 8\\
79 & 1\\
80 & ?\\
81 & ?\\
82 & 2\\
83 & 1\\
84 & 44\\
85 & 1\\
86 & 2\\
87 & 1\\
88 & 39\\
89 & 1\\
90 & 18\\
91 & 1\\
92 & 7\\
93 & 3\\
94 & 2\\
95 & 1\\
96 & ?\\
97 & 1\\
98 & 17\\ 
99 & 4 \\
100 & 100 \\
101 & 1 \\
102 & 4 \\
103 & 1 \\
104 & 52 \\
105 & 3 \\
106 & 2 \\
107 & 1 \\
 & \\
  & \\
\hline
    \end{tabular}
    \caption{Numbers of indecomposable solutions}
    \label{table:IndSols}
\end{table}

\subsection{Results}
A summary of our enumeration results are displayed in Table \ref{table:IndSols}. More detailed results can be found at \cite{GITPAGE}. For a positive integer $n$ listed in the first column, the second column lists the total number of indecomposable solutions, up to isomorphism, with permutation brace of size $n$. The results are given up to $n=107$ with exceptions at $n=32,80,81$ and $96$. As noted at the beginning of this section, the main bottleneck in these computations appears to be the computation of the brace automorphism groups.

\enlargethispage{\baselineskip}

The full version of our results contains lists of matrices which can be read into {\sc Gap} to produce explicit solutions, and thus we do not simply enumerate these solutions but construct them. We note that the results presented in Table \ref{table:IndSols} agree with the relevant parts of the literature. In particular, we recover the results of Section \ref{IndSols_p3} in the case that $p=3$; that is, up to isomorphism, there are
\[(3^3+3^2+11\cdot 3+3)/4 + (3^3+ 3\cdot 3)/4=27\]
solutions with permutation brace of size $3^3=27$. All experiments were performed on a machine with an AMD Ryzen(TM) 7 7700X CPU running Ubuntu 24.04.1 LTS with Linux kernel 6.8.0-85-generic, using 128GB RAM.

\section*{Acknowledgments.}
The authors would like to thank Carsten Dietzel, Silvia Properzi and Leandro Vendramin for fruitful discussions.

This work was supported by Fonds voor Wetenschappelijk Onderzoek G004124N and Vrije Universiteit Brussel OZR3762.

Wiertel was also supported by the Polish National Agency for Academic
Exchange within Bekker Programme BPN/BEK/2024/1/00311/U/00001.


\bibliography{MyBib}

@misc{GITPAGE,
    AUTHOR = {Darlington, A. and Wiertel, M.},
    year = {2026}, 
    TITLE = {Indecomposable solutions},
    note={\url{https://github.com/Andrew-Darlington/Indecomposable-Solutions}},
}

@manual{GAP4,
    organization = "The GAP~Group",
    title        = "{GAP -- Groups, Algorithms, and Programming,
                    Version 4.16.0}",
    year         = 2026,
    url          = "\url{https://www.gap-system.org}",
    }

@article {Yan67,
    AUTHOR = {Yang, C. N.},
     TITLE = {Some exact results for the many-body problem in one dimension
              with repulsive delta-function interaction},
   JOURNAL = {Phys. Rev. Lett.},
  FJOURNAL = {Physical Review Letters},
    VOLUME = {19},
      YEAR = {1967},
     PAGES = {1312--1315},
      ISSN = {0031-9007},
   MRCLASS = {81.20},
MRREVIEWER = {S.\ Deser},
       DOI = {10.1103/PhysRevLett.19.1312},
       URL = {https://doi.org/10.1103/PhysRevLett.19.1312},
}

@article {Bax72,
    AUTHOR = {Baxter, R. J.},
     TITLE = {Partition function of the eight-vertex lattice model},
   JOURNAL = {Ann. Physics},
  FJOURNAL = {Annals of Physics},
    VOLUME = {70},
      YEAR = {1972},
     PAGES = {193--228},
      ISSN = {0003-4916,1096-035X},
   MRCLASS = {82.46},
MRREVIEWER = {S.\ Sherman},
       DOI = {10.1016/0003-4916(72)90335-1},
       URL = {https://doi.org/10.1016/0003-4916(72)90335-1},
}

@incollection {Dri92,
    AUTHOR = {Drinfel'd, V. G.},
     TITLE = {On some unsolved problems in quantum group theory},
 BOOKTITLE = {Quantum groups ({L}eningrad, 1990)},
    SERIES = {Lecture Notes in Math.},
    VOLUME = {1510},
     PAGES = {1--8},
 PUBLISHER = {Springer, Berlin},
      YEAR = {1992},
      ISBN = {3-540-55305-3},
   MRCLASS = {17B37 (16W30 81R50)},
MRREVIEWER = {Yvette\ Kosmann-Schwarzbach},
       DOI = {10.1007/BFb0101175},
       URL = {https://doi.org/10.1007/BFb0101175},
}

@article {Ram23,
    AUTHOR = {Ram\'irez, S.},
     TITLE = {Indecomposable solutions of the {Y}ang--{B}axter equation with
              permutation group of sizes {$pq$} and {$p^2q$}},
   JOURNAL = {Comm. Algebra},
  FJOURNAL = {Communications in Algebra},
    VOLUME = {51},
      YEAR = {2023},
    NUMBER = {10},
     PAGES = {4185--4194},
      ISSN = {0092-7872,1532-4125},
   MRCLASS = {16T25},
MRREVIEWER = {Luz\ Adriana\ Mej\'ia Casta\~no},
       DOI = {10.1080/00927872.2023.2200827},
       URL = {https://doi.org/10.1080/00927872.2023.2200827},
}

@article {Bac15,
    AUTHOR = {Bachiller, D.},
     TITLE = {Classification of braces of order {$p^3$}},
   JOURNAL = {J. Pure Appl. Algebra},
  FJOURNAL = {Journal of Pure and Applied Algebra},
    VOLUME = {219},
      YEAR = {2015},
    NUMBER = {8},
     PAGES = {3568--3603},
      ISSN = {0022-4049,1873-1376},
   MRCLASS = {81Q05 (16T25)},
MRREVIEWER = {Theo\ Johnson-Freyd},
       DOI = {10.1016/j.jpaa.2014.12.013},
       URL = {https://doi.org/10.1016/j.jpaa.2014.12.013},
}

@article {BCJ16,
    AUTHOR = {Bachiller, D. and Ced\'o, F. and Jespers, E.},
     TITLE = {Solutions of the {Y}ang--{B}axter equation associated with a
              left brace},
   JOURNAL = {J. Algebra},
  FJOURNAL = {Journal of Algebra},
    VOLUME = {463},
      YEAR = {2016},
     PAGES = {80--102},
      ISSN = {0021-8693,1090-266X},
   MRCLASS = {16T25 (20B25)},
MRREVIEWER = {Leandro\ Vendramin},
       DOI = {10.1016/j.jalgebra.2016.05.024},
       URL = {https://doi.org/10.1016/j.jalgebra.2016.05.024},
}

@article {DPT25,
    AUTHOR = {Dietzel, C. and Properzi, S. and Trappeniers, S.},
     TITLE = {Indecomposable involutive set-theoretical solutions to the
              {Y}ang--{B}axter equation of size {$p^2$}},
   JOURNAL = {Comm. Algebra},
  FJOURNAL = {Communications in Algebra},
    VOLUME = {53},
      YEAR = {2025},
    NUMBER = {3},
     PAGES = {1238--1256},
      ISSN = {0092-7872,1532-4125},
   MRCLASS = {16T25 (20N02 81R50)},
MRREVIEWER = {Jo\~ao\ Matheus Jury Giraldi},
       DOI = {10.1080/00927872.2024.2405024},
       URL = {https://doi.org/10.1080/00927872.2024.2405024},
}

@article {JPZ-D,
    AUTHOR = {Jedli\v{c}ka, P. and Pilitowska, A. and
              Zamojska-Dzienio, A.},
     TITLE = {Cocyclic braces and indecomposable cocyclic solutions of the
              {Y}ang--{B}axter equation},
   JOURNAL = {Proc. Amer. Math. Soc.},
  FJOURNAL = {Proceedings of the American Mathematical Society},
    VOLUME = {150},
      YEAR = {2022},
    NUMBER = {10},
     PAGES = {4223--4239},
      ISSN = {0002-9939,1088-6826},
   MRCLASS = {16T25 (20B35 20D15)},
       DOI = {10.1090/proc/15962},
       URL = {https://doi.org/10.1090/proc/15962},
}

@article{ESchSol,
author = {P. Etingof and T. Schedler and A. Soloviev},
title = {{Set-theoretical solutions to the quantum {Y}ang--{B}axter equation}},
volume = {100},
journal = {Duke Mathematical Journal},
number = {2},
publisher = {Duke University Press},
pages = {169 -- 209},
year = {1999},
doi = {10.1215/S0012-7094-99-10007-X},
URL = {https://doi.org/10.1215/S0012-7094-99-10007-X}
}

@article {Rump21,
    AUTHOR = {Rump, W.},
     TITLE = {Cocyclic solutions to the {Y}ang--{B}axter equation},
   JOURNAL = {Proc. Amer. Math. Soc.},
  FJOURNAL = {Proceedings of the American Mathematical Society},
    VOLUME = {149},
      YEAR = {2021},
    NUMBER = {2},
     PAGES = {471--479},
      ISSN = {0002-9939,1088-6826},
   MRCLASS = {05E18 (08A05 16T25 68R05 81R50)},
       DOI = {10.1090/proc/15220},
       URL = {https://doi.org/10.1090/proc/15220},
}

@article {Cas23,
    AUTHOR = {Castelli, M.},
     TITLE = {Classification of uniconnected involutive solutions of the
              {Y}ang--{B}axter equation with odd size and a {Z}-group
              permutation group},
   JOURNAL = {Int. Math. Res. Not. IMRN},
  FJOURNAL = {International Mathematics Research Notices. IMRN},
      YEAR = {2023},
    NUMBER = {14},
     PAGES = {11962--11985},
      ISSN = {1073-7928,1687-0247},
   MRCLASS = {16T25 (20B25)},
       DOI = {10.1093/imrn/rnac185},
       URL = {https://doi.org/10.1093/imrn/rnac185},
}

@article {CR24,
    AUTHOR = {Castelli, M. and Ram\'irez, S.},
     TITLE = {On uniconnected solutions of the {Y}ang--{B}axter equation and
              {D}ehornoy's class},
   JOURNAL = {J. Algebra},
  FJOURNAL = {Journal of Algebra},
    VOLUME = {657},
      YEAR = {2024},
     PAGES = {57--80},
      ISSN = {0021-8693,1090-266X},
   MRCLASS = {16T25 (81R50)},
MRREVIEWER = {Jialei\ Chen},
       DOI = {10.1016/j.jalgebra.2024.04.032},
       URL = {https://doi.org/10.1016/j.jalgebra.2024.04.032},
}

@article {JP22,
    AUTHOR = {Jedli\v{c}ka, P. and Pilitowska, A.},
     TITLE = {Indecomposable involutive solutions of the {Y}ang--{B}axter
              equation of multipermutation level 2 with non-abelian
              permutation group},
   JOURNAL = {J. Combin. Theory Ser. A},
  FJOURNAL = {Journal of Combinatorial Theory. Series A},
    VOLUME = {197},
      YEAR = {2023},
     PAGES = {Paper No. 105753, 35},
      ISSN = {0097-3165,1096-0899},
   MRCLASS = {16T25},
       DOI = {10.1016/j.jcta.2023.105753},
       URL = {https://doi.org/10.1016/j.jcta.2023.105753},
}

@article {JPZ-D21,
    AUTHOR = {Jedli\v{c}ka, P. and Pilitowska, A. and
              Zamojska-Dzienio, A.},
     TITLE = {Indecomposable involutive solutions of the {Y}ang--{B}axter
              equation of multipermutational level 2 with abelian
              permutation group},
   JOURNAL = {Forum Math.},
  FJOURNAL = {Forum Mathematicum},
    VOLUME = {33},
      YEAR = {2021},
    NUMBER = {5},
     PAGES = {1083--1096},
      ISSN = {0933-7741,1435-5337},
   MRCLASS = {16T25 (20B35)},
MRREVIEWER = {Leandro\ Vendramin},
       DOI = {10.1515/forum-2021-0130},
       URL = {https://doi.org/10.1515/forum-2021-0130},
}

@article {COk23,
    AUTHOR = {Ced\'o, F. and Okni\'nski, J.},
     TITLE = {Indecomposable solutions of the {Y}ang--{B}axter equation of
              square-free cardinality},
   JOURNAL = {Adv. Math.},
  FJOURNAL = {Advances in Mathematics},
    VOLUME = {430},
      YEAR = {2023},
     PAGES = {Paper No. 109221, 26},
      ISSN = {0001-8708,1090-2082},
   MRCLASS = {16T25 (20B15 20F16)},
       URL = {https://doi.org/10.1016/j.aim.2023.109221},
}

@misc{ YangBaxter,
  author =           {Vendramin, L. and Konovalov, O.},
  title =            {{YangBaxter},  {Combinatorial  Solutions  for  the {Y}ang--{B}axter equation},
                      {V}ersion 0.10.7},
  month =            {Jul},
  year =             {2025},
  note =             {GAP package},
  howpublished =     {\url{https://gap-packages.github.io/YangBaxter}},
  printedkey =       {VK25}
}

@article {GV17,
	author = {Guarnieri, L. and Vendramin, L.},
	title = {Skew braces and the {Y}ang--{B}axter equation},
	year = {2017},
	journal = {Mathematics of Computation},
	volume = {86},
	number = {307},
	pages = {2519 – 2534},
	doi = {10.1090/mcom/3161},
	url = {https://www.scopus.com/inward/record.uri?eid=2-s2.0-85018473062&doi=10.1090%2fmcom%2f3161&partnerID=40&md5=7bf6db2c61609d671119adbfa48cd493},
	type = {Article}
}

@article {LV24,
    AUTHOR = {Lebed, V. and Ram\'irez, S. and Vendramin,
              L.},
     TITLE = {Involutive {Y}ang--{B}axter: cabling, decomposability, and
              {D}ehornoy class},
   JOURNAL = {Rev. Mat. Iberoam.},
  FJOURNAL = {Revista Matem\'atica Iberoamericana},
    VOLUME = {40},
      YEAR = {2024},
    NUMBER = {2},
     PAGES = {623--635},
      ISSN = {0213-2230,2235-0616},
   MRCLASS = {16T25 (08A05)},
       DOI = {10.4171/rmi/1438},
       URL = {https://doi.org/10.4171/rmi/1438},
}

@article {RV22,
    AUTHOR = {Ram\'irez, S. and Vendramin, L.},
     TITLE = {Decomposition theorems for involutive solutions to the
              {Y}ang--{B}axter equation},
   JOURNAL = {Int. Math. Res. Not. IMRN},
  FJOURNAL = {International Mathematics Research Notices. IMRN},
      YEAR = {2022},
    NUMBER = {22},
     PAGES = {18078--18091},
      ISSN = {1073-7928,1687-0247},
   MRCLASS = {16T25 (05E10)},
MRREVIEWER = {Marco\ Castelli},
       DOI = {10.1093/imrn/rnab232},
       URL = {https://doi.org/10.1093/imrn/rnab232},
}

@article {AMV22,
    AUTHOR = {Akg\"un, \"O. and Mereb, M. and Vendramin, L.},
     TITLE = {Enumeration of set-theoretic solutions to the {Y}ang--{B}axter
              equation},
   JOURNAL = {Math. Comp.},
  FJOURNAL = {Mathematics of Computation},
    VOLUME = {91},
      YEAR = {2022},
    NUMBER = {335},
     PAGES = {1469--1481},
      ISSN = {0025-5718,1088-6842},
   MRCLASS = {16T25},
       DOI = {10.1090/mcom/3696},
       URL = {https://doi.org/10.1090/mcom/3696},
}

@article {R05,
    AUTHOR = {Rump, W.},
     TITLE = {A decomposition theorem for square-free unitary solutions of
              the quantum {Y}ang--{B}axter equation},
   JOURNAL = {Adv. Math.},
  FJOURNAL = {Advances in Mathematics},
    VOLUME = {193},
      YEAR = {2005},
    NUMBER = {1},
     PAGES = {40--55},
      ISSN = {0001-8708,1090-2082},
   MRCLASS = {81R50 (16W35 81R12)},
       DOI = {10.1016/j.aim.2004.03.019},
       URL = {https://doi.org/10.1016/j.aim.2004.03.019},
}

@article {CCP19,
    AUTHOR = {Castelli, M. and Catino, F. and Pinto, G.},
     TITLE = {Indecomposable involutive set-theoretic solutions of the
              {Y}ang--{B}axter equation},
   JOURNAL = {J. Pure Appl. Algebra},
  FJOURNAL = {Journal of Pure and Applied Algebra},
    VOLUME = {223},
      YEAR = {2019},
    NUMBER = {10},
     PAGES = {4477--4493},
      ISSN = {0022-4049,1873-1376},
   MRCLASS = {16T25 (20E22 20F16)},
       DOI = {10.1016/j.jpaa.2019.01.017},
       URL = {https://doi.org/10.1016/j.jpaa.2019.01.017},
}

@article {C25,
    AUTHOR = {Castelli, M.},
     TITLE = {On the indecomposable involutive solutions of the
              {Y}ang--{B}axter equation of finite primitive level},
   JOURNAL = {Publ. Mat.},
  FJOURNAL = {Publicacions Matem\`atiques},
    VOLUME = {69},
      YEAR = {2025},
    NUMBER = {2},
     PAGES = {429--444},
      ISSN = {0214-1493,2014-4350},
   MRCLASS = {16T25 (20N02 81R50)},
       DOI = {10.5565/publmat6922509},
       URL = {https://doi.org/10.5565/publmat6922509},
}

@article {Rump07,
    AUTHOR = {Rump, W.},
     TITLE = {Braces, radical rings, and the quantum {Y}ang--{B}axter
              equation},
   JOURNAL = {J. Algebra},
  FJOURNAL = {Journal of Algebra},
    VOLUME = {307},
      YEAR = {2007},
    NUMBER = {1},
     PAGES = {153--170},
      ISSN = {0021-8693,1090-266X},
   MRCLASS = {16Y99 (16W30)},
       DOI = {10.1016/j.jalgebra.2006.03.040},
       URL = {https://doi.org/10.1016/j.jalgebra.2006.03.040},
}

@article {CJO14,
    AUTHOR = {Ced\'o, F. and Jespers, E. and Okni\'nski, J.},
     TITLE = {Braces and the {Y}ang--{B}axter equation},
   JOURNAL = {Comm. Math. Phys.},
  FJOURNAL = {Communications in Mathematical Physics},
    VOLUME = {327},
      YEAR = {2014},
    NUMBER = {1},
     PAGES = {101--116},
      ISSN = {0010-3616,1432-0916},
   MRCLASS = {81Q05 (16T25)},
       DOI = {10.1007/s00220-014-1935-y},
       URL = {https://doi.org/10.1007/s00220-014-1935-y},
}

@phdthesis{NZ18,
  author = {Nejabati Zenouz, K.},
  title  = {On {H}opf-{G}alois Structures and Skew Braces of Order $p^3$},
  school = {the University of Exeter},
  type   = {PhD thesis},
  year   = {2018},
  url    = {https://ore.exeter.ac.uk/repository/handle/10871/32248},
  note   = {https://ore.exeter.ac.uk/repository/handle/10871/32248},
}

@article {GI18,
    AUTHOR = {Gateva-Ivanova, T.},
     TITLE = {Set-theoretic solutions of the {Y}ang--{B}axter equation,
              braces and symmetric groups},
   JOURNAL = {Adv. Math.},
  FJOURNAL = {Advances in Mathematics},
    VOLUME = {338},
      YEAR = {2018},
     PAGES = {649--701},
      ISSN = {0001-8708,1090-2082},
   MRCLASS = {16T25 (16S37 16W22 16W50 20B35 81R50 81R60)},
MRREVIEWER = {Leandro\ Vendramin},
       DOI = {10.1016/j.aim.2018.09.005},
       URL = {https://doi.org/10.1016/j.aim.2018.09.005},
}

@article {CSV19,
    AUTHOR = {Ced\'o, F. and Smoktunowicz, A. and Vendramin, L.},
     TITLE = {Skew left braces of nilpotent type},
   JOURNAL = {Proc. Lond. Math. Soc. (3)},
  FJOURNAL = {Proceedings of the London Mathematical Society. Third Series},
    VOLUME = {118},
      YEAR = {2019},
    NUMBER = {6},
     PAGES = {1367--1392},
      ISSN = {0024-6115,1460-244X},
   MRCLASS = {16T25 (81R50 82B10)},
MRREVIEWER = {Paola\ Stefanelli},
       DOI = {10.1112/plms.12209},
       URL = {https://doi.org/10.1112/plms.12209},
}

\end{document}